\documentclass{article}
\usepackage{amsmath,amssymb, amsthm,amsfonts}
\usepackage[dvipsnames]{xcolor}
\usepackage[titletoc,title]{appendix}
\numberwithin{equation}{section}
\newtheorem{theorem}{Theorem}[section]
\newtheorem{lemma}{Lemma}[section]
\newtheorem{remark}{Remark}[section]
\newtheorem{corollary}{Corollary}[section]
\newtheorem{assumption}{Assumption}[section]
\usepackage{graphicx} 
\newcommand{\bfv}{\mathbf{v}}
\newcommand{\bfu}{\mathbf{u}}
\newcommand{\bfU}{\mathbf{U}}

\title{Long-Time Nonlinear Stability of Large-Amplitude Radial Rarefactive Flows for the Multidimensional Compressible Euler Equations}

\author{
Faris El-Katri\thanks{Department of Mathematics, University of Kansas, Lawrence, KS 66045 (elkatri@ku.edu).}
}

\date{}

\begin{document}

\maketitle

\begin{abstract}
In this paper, we study multidimensional perturbations of radially symmetric rarefactive flows for the compressible Euler equations. No assumptions on the smallness of the background are imposed. We derive lower bounds on both the classical lifespan and stability timespan of the perturbed solution. These lower bounds are demonstrated to tend to infinity as the size of the perturbation approaches zero. Furthermore, owing to the rarefactive nature of the background, the resulting estimate is shown to be substantially longer than the corresponding estimate for a general classical background with bounded first derivatives. 
\end{abstract}

\noindent
\textbf{2020 Mathematical Subject Classification:} 35L65, 35L67, 76N15.

\medskip

\noindent
\textbf{Key words:} Compressible Euler equations, rarefactive flows, nonlinear stability, lifespan estimates, radial symmetry, point vacuum.

\tableofcontents

\section{Introduction}

This paper examines the multidimensional stability of radially symmetric rarefactive solutions to the compressible Euler equations, 

\begin{align}\label{Eqn:Euler_Full_tensor}
\begin{aligned}
&\partial_t \rho+\nabla\cdot(\rho \mathbf u)=0,\\[0.5em]
&\partial_t(\rho \mathbf u)
+\nabla\cdot(\rho \mathbf u\otimes \mathbf u+pI)=0,\\[0.5em]
&\partial_t E+\nabla\cdot\big((E+p)\mathbf u\big)=0. 
\end{aligned}
\end{align}
The pressure, specific internal energy, and total energy density are given by

\begin{align*}
    p&:=Ke^S\rho^\gamma, 
    \\
    e&:=\frac{p}{(\gamma-1)\rho}, 
    \\
    E&:=\rho e+\frac{1}{2}\rho |\mathbf u|^2, 
\end{align*} 
where $K>0$, $1<\gamma<3$, and $S$ denotes the normalized specific entropy. The compressible Euler equations are one of the quintessential examples of quasilinear hyperbolic systems of conservation laws. Their characteristic structure was central to the classical theory of nonlinear hyperbolic conservation laws \cite{L_1,CF} and remains essential to the modern analysis of hyperbolic systems \cite{B_1,D_1,M_1}. 

A principal concern in the study of the Euler equations is the breakdown of classical solutions and their weak continuation past the time of singularity formation. Early results on shock formation focused on the case of one spatial dimension \cite{A1,J_1}. Later work in the multidimensional setting established the breakdown of classical solutions primarily through energy-based methods \cite{S_1,S_2}. On the other hand, it was shown that small perturbations of size $\varepsilon$ around a constant state have a classical lifespan satisfying $T^\varepsilon_{\mathrm{life}}\gtrsim\varepsilon^{-1}$ \cite{S_2}. More recent work by Christodoulou gave rise to the geometric formulation of shock formation in three spatial dimensions \cite{C_1}. Notably, this allowed for the analysis of shock formation arising from perturbations of simple plane waves \cite{LS}. Later work constructed stable point shocks without symmetry assumptions using modulated self-similar variables \cite{BSV_1}. In particular, it was shown that, for a class of compressions of size $1/\varepsilon$ around a constant state, the compressible Euler equations generically form a $C^{1/3}$ cusp within a time of order $\varepsilon$. 

Paradoxically, while smooth initial data may evolve into shock formation in finite time, it is also the case that discontinuities in initial data may be resolved by continuous rarefaction fans. In contrast to the compressive mechanism inducing shock formation, rarefaction waves correspond to expansive families of characteristics. They arise in the one-dimensional Riemann problem, where jump discontinuities give rise to rarefaction fans instead of shocks. Their stability has been examined extensively in both viscous and inviscid systems. The stability of rarefaction wave solutions to the compressible Navier--Stokes equations was established in \cite{LZ_1,NYZ_1}. Later works analyzed the stability and uniqueness of rarefaction waves in the multidimensional setting within classes of entropy solutions \cite{CC_1,FK_1}, while more recent work proved their nonlinear stability under multidimensional perturbations using energy methods \cite{LY1,LY2}. 

Parallel to the study of shock formation is the problem of characterizing which classes of initial data lead to singularity formation. In one spatial dimension, data along each characteristic family may be classified as compressive if the derivative of the corresponding Riemann invariant is negative somewhere, and rarefactive otherwise. Owing to the time-dependent lower bound on the density \cite{Chen_1}, it was shown via non-constructive methods that singularity formation in the one-dimensional isentropic Euler equations occurs if and only if the initial data contain a compressive region, regardless of its magnitude \cite{CPZ_1}. While a complete characterization is readily available in the one-dimensional case, the spherically symmetric case contains additional structure. Indeed, owing largely to the radial source terms, a na\"ive approach based directly on the derivatives of the Riemann variables is not successful. However, by requiring that the stationary solution be neither compressive nor rarefactive, the authors in \cite{CEHS_1} derived modified gradient variables whose signs encode the desired compressive and rarefactive regimes. In that setting, it was shown that sufficiently strong compression leads to singularity formation. This compressive aspect is closely related to earlier work on radially symmetric expanding flows \cite{CCW}. The principal rarefactive result of \cite{CEHS_1} is that initial data rarefactive with respect to both modified gradient variables give rise to global-in-time classical solutions without any smallness assumption on the amplitude of the data, thus avoiding both shock and implosion singularities. 

Related global smooth solutions to the compressible Euler equations were constructed by Grassin for small initial density by imposing conditions on the spectrum of the initial velocity gradient \cite{Grassin}. In contrast, \cite{CEHS_1} allows for arbitrarily large amplitude initial data, and the stability of the corresponding rarefactive profiles under multidimensional perturbations remains largely unexplored. Indeed, one of the two rarefactive profiles in \cite{CEHS_1} contains a point vacuum at the origin, whose behavior under perturbation introduces additional complications. In this paper, we prove a lower bound on the nonlinear stability timespan of large-amplitude rarefactive profiles under multidimensional perturbations. Remarkably, we demonstrate that the expansive structure of the underlying rarefactive profile improves the stability estimates beyond those obtained for perturbations of a general classical background with bounded first derivatives. 

More precisely, let $(\bar h, \bar\bfu)$ denote a radially symmetric rarefactive flow satisfying the assumptions of \cite{CEHS_1}. We consider multidimensional perturbations $(\sigma_0,\bfv_0,\theta_0)$ of size $\varepsilon$ in $H^N(\mathbb{R}^3)$ and define the perturbed initial data by 

\begin{align*}
    &H_0:=\bar h_0+\sigma_0 , 
    \\
    &\bfU_0:=\bar\bfu_0+\bfv_0, 
\end{align*} 
where the initial sonic speed, $H_0$, is strictly positive everywhere except possibly at the origin. We then prove that, for $\varepsilon\ll1$, the perturbed flow remains classical and close to the radial background up to a time $T^\varepsilon$ satisfying 

\begin{align*}
    T^\varepsilon_{\mathrm{life}}\geq T^\varepsilon
    \gtrsim
    \left(\log\frac{1}{\varepsilon}
    \right)^{\mathcal C_N}.
\end{align*} 
We further note that $T^\varepsilon$ tends to infinity as $\varepsilon\rightarrow0$, and the exponent $\mathcal{C}_N$ is determined by the derivative bounds of the rarefactive background. While this estimate is weaker than the $1/\varepsilon$ lower bound for perturbations of a constant state, there is a marked difference between the two regimes. The $1/\varepsilon$ estimate is sharp, in general, for perturbations with vorticity and cannot be improved without additional assumptions \cite{LS}. This work, by comparison, perturbs a class of arbitrarily large data shown to produce global classical solutions. These data are free from impositions of self-similarity or affine structure. Furthermore, owing to their rarefactive structure, the derived estimate is substantially stronger than the general $\log\log\frac{1}{\varepsilon}$ timespan obtained by applying the energy methods contained herein to perturbations of an arbitrary global classical solution with uniformly bounded-in-time first derivatives. 

The proof will proceed as follows. We first refine the derivative estimates for the radially symmetric rarefactive background. The modified gradient variables introduced in \cite{CEHS_1} yield decay of the first derivatives at a rate of $\mathcal{O}(1+t)^{-1}$. We then examine the characteristic equations for higher-order derivatives and show inductively that the $m$-th order derivatives of the background grow at most polynomially in time, 

\begin{align*}
    \|\nabla^m \bar h(t,\cdot)\|_{L^\infty}
    +\|\nabla^m \bar{\mathbf u}(t,\cdot)\|_{L^\infty}
    \lesssim (1+t)^{B_m},
\end{align*}
for constants $B_m>0$. We then derive the perturbation equations and place them in symmetric-hyperbolic form. Noting that the symmetrizer remains uniformly coercive even in the presence of vacuum, Kato's theorem \cite{Kato1975} yields local existence. We then define a cumulative Sobolev energy and prove its coercivity under a mild bootstrap assumption. Combining this with the background derivative estimates and Sobolev embedding, we obtain the differential inequality 

\begin{align*}
    \frac{d}{dt}\mathcal{E}_N(t) 
    \lesssim 
    (1+t)^{2B_{N+1}^*} 
    \mathcal{E}_N(t)\left(  
    1+\mathcal{E}_N(t)^{1/2}\right)
    +\mathcal{E}_N(t)^2. 
\end{align*} 
Gr\"onwall's inequality then gives us the desired stability estimate up to a time at least of order 

\begin{align*}
    \left(\log\frac1\varepsilon\right)^{\mathcal C_N}. 
\end{align*} 
Combining this a priori estimate with Kato's local existence theory and a standard continuation argument proves the claimed lower bound on the maximal lifespan of the solution. 

The remainder of the paper is organized as follows. In Section 2, we analyze the radially symmetric rarefactive background. We first recall the relevant results from \cite{CEHS_1}. These results are refined to quantify the rate of decay of the modified gradient variables. Once the background has been fully introduced, we provide a rigorous statement of our main results. The remainder of this section is dedicated to obtaining time-dependent upper bounds on the higher-order derivatives of the radially symmetric rarefactive background. 

In Section 3, we develop the framework necessary for analyzing the evolution of the perturbed system. Using our definitions of the perturbation variables and the fact that the radially symmetric rarefactive background solves (\ref{Eqn:Euler_Full_tensor}), we obtain an equivalent system of evolution equations for the perturbation variables. We then define an appropriate Sobolev energy and show that it is coercive with respect to the $H^N(\mathbb{R}^3)$ norm of the perturbation. The remainder of this section is devoted to obtaining the desired closed differential inequality for the cumulative energy. This is obtained primarily through commutator estimates and Sobolev embedding. 

In Section 4, we close our bootstrap assumptions and prove the desired stability timespan estimate. From standard local existence theory, it is known that symmetric-hyperbolic systems have some positive lifespan under relatively mild conditions. Thus, the first part of Section 4 is dedicated to establishing local existence even in the presence of vacuum at the origin and infinity. Then, using the differential inequality derived in Section 3, we calculate the claimed stability timespan. Finally, using a standard continuation argument, we demonstrate classical existence until this time as well. This completes the proof of the main theorem.

\section{The Radial Rarefactive Background} 

First, we summarize existing results in the literature. In particular, Lemmas \ref{Thm:SupersonicInvariance}--\ref{Thm:wUpperBound} and Theorem \ref{Thm:Radial_Rarefactive_Existence} were proven in \cite{CEHS_1}. We first seek to rewrite (\ref{Eqn:Euler_Full_tensor}) in a form better suited for our calculations. Classical solutions to (\ref{Eqn:Euler_Full_tensor}) equivalently solve 

\begin{align}\label{Eqn:Euler_Primitive}
\begin{aligned}
    &\rho_t+\mathbf u\cdot\nabla\rho
    +\rho\nabla\cdot\mathbf u=0,
    \\[0.5em]
    &\mathbf u_t+(\mathbf u\cdot\nabla)\mathbf u
    +\frac{1}{\rho}\nabla p=0,
    \\[0.5em]
    &S_t+\mathbf u\cdot\nabla S=0.
\end{aligned}
\end{align} 
Furthermore, it is convenient to replace the density, $\rho$, with the sonic speed, $h$, where 

\begin{align*}
    h:=\sqrt{p_\rho}. 
\end{align*} 
This yields 

\begin{align}\label{Eqn:Euler_huS}
\begin{aligned}
    &h_t+\mathbf u\cdot\nabla h
    +ah\nabla\cdot\mathbf u=0,
    \\[0.5em]
    &\mathbf u_t+(\mathbf u\cdot\nabla)\mathbf u
    +\frac{1}{a}h\nabla h
    -\frac{1}{2a\gamma}h^2\nabla S=0,
    \\[0.5em]
    &S_t+\mathbf u\cdot\nabla S=0. 
\end{aligned}
\end{align} 
where 

\begin{align*}
    a:=\frac{\gamma-1}{2}. 
\end{align*}


Under spherical symmetry and constant entropy, (\ref{Eqn:Euler_huS}) reduces to a system of two equations over the half-line 

\begin{align*}
\begin{aligned}
    &h_t+uh_r+\frac{\gamma-1}{2}hu_r 
    =-\frac{(\gamma-1)uh}{r}, 
    \\
    &u_t+uu_r+\frac{2}{\gamma-1}hh_r=0,  
\end{aligned}
\end{align*} 
where 

\begin{align*}
    &\bfu(t,x)=u(t,r)e_r, 
    \\
    &h(t,x)=h(t,r), 
    \\
    &r=|x|. 
\end{align*}
This system may be diagonalized with eigenvalues $c_1,c_2$ and Riemann variables $w,z$. 

\begin{align*}
    &z:=u-\frac{2}{\gamma-1}h, 
    \\
    &w:=u+\frac{2}{\gamma-1}h, 
    \\
    &c_1:=u-h=\frac{3-\gamma}{4}w 
    +\frac{\gamma+1}{4}z, 
    \\
    &c_2:=u+h=\frac{\gamma+1}{4}w 
    +\frac{3-\gamma}{4}z. 
\end{align*}
These, in turn, give rise to directional derivatives over which our Riemann variables are transported, 

\begin{align}
    &\partial_1:=\partial_t+c_1\partial_r, 
    \label{Eqn:partial_1}
    \\
    &\partial_2:=\partial_t+c_2\partial_r, 
    \label{Eqn:partial_2}
    \\
    &\partial_1z=\frac{\gamma-1}{4r}\left( 
    w^2-z^2\right), 
    \label{Eqn:partial_1z}
    \\
    &\partial_2w=-\frac{\gamma-1}{4r}\left( 
    w^2-z^2\right). 
    \label{Eqn:partial_2w}
\end{align} 
For later convenience, we also define flow maps. 

\begin{align*}
    &\psi_t:=c_1\circ\psi ,
    \\
    &\eta_t:=c_2\circ\eta ,
    \\
    &\psi(0,y)=\eta(0,y)=y .
\end{align*} 
These allow us to rewrite Equations (\ref{Eqn:partial_1})--(\ref{Eqn:partial_2w}), 

\begin{align*}
    &\partial_t(z\circ\psi)=\left(\frac{\gamma-1}{4r} 
    (w^2-z^2)\right) \circ\psi, 
    \\
    &\partial_t(w\circ\eta)=\left(-\frac{\gamma-1}{4r}(w^2-z^2)\right)\circ\eta. 
\end{align*} 
Using these, we define the domain of a given characteristic triangle, 

\begin{align*} 
    &\Omega_t([a,b]):=\eta(t,[a,b])\cap\psi(t,[a,b]), 
    \\
    &\Omega(T,[a,b]):= \bigcup_{0\leq t\leq T} 
    \{t\}\times \Omega_t([a,b]). 
\end{align*}
We then define modified gradient variables, $\alpha,\beta$, 

\begin{align*}
    &\beta:=z_r-\frac{2uh}{rc_1} ,
    \\
    &\alpha:=w_r+\frac{2uh}{rc_2} .
\end{align*} 
The signs of these gradient variables determine the wave character. A profile is rarefactive in the $1$- and $2$-characteristic families when $\beta\geq0$ and $\alpha\geq0$, respectively, and compressive in these families when $\beta<0$ and $\alpha<0$, respectively. For the supersonic expanding wave, it was shown in \cite{CEHS_1} that the fully rarefactive and fully compressive regimes are preserved under the classical flow. Specifically, solutions with initial data rarefactive in both gradient variables remain rarefactive for all time, and solutions with initial data compressive in both gradient variables remain compressive during the lifespan of the classical solution. These variables evolve along the $1$- and $2$-characteristics, respectively, according to  

\begin{align*}
    &\partial_1\beta=-\frac{\gamma+1}{4}\beta^2 
    -\frac{3-\gamma}{4}\alpha\beta 
    +A_1\alpha-B_1\beta,
    \\
    &\partial_2\alpha=-\frac{\gamma+1}{4}\alpha^2 
    -\frac{3-\gamma}{4}\alpha\beta 
    +A_2\beta-B_2\alpha, 
\end{align*} 
with variable coefficients given by 

\begin{align*}
    &A_1:=\frac{c_2}{rc_1^2}\left( 
    \frac{\gamma-1}{2}u^2-h^2\right),  
    \\
    &A_2:=\frac{c_1}{rc_2^2}\left(\frac{\gamma-1}{2}
    u^2-h^2\right), 
    \\
    &B_1:=\frac{2}{rc_1^2}\left( 
    \frac{\gamma-1}{4}u^3-\frac{1}{2}h^3 
    -\frac{\gamma-1}{4}u^2h 
    +\frac{1}{2}uh^2 
    +\frac{uhc_1}{c_2}\left(h+\frac{\gamma-1}{2}u \right)\right), 
    \\
    &B_2:=\frac{2}{rc_2^2}\left( 
    \frac{\gamma-1}{4}u^3+\frac{1}{2}h^3 
    +\frac{\gamma-1}{4}u^2h 
    +\frac{1}{2}uh^2 
    +\frac{uhc_2}{c_1}\left(h-\frac{\gamma-1}{2}u \right)\right). 
\end{align*} 
As shown in \cite{CEHS_1}, whenever $u\geq\frac{2}{\gamma-1}h$, these coefficients are nonnegative and satisfy 

\begin{align*}
    &A_i,B_i\geq 0, \\
    &B_i-A_i\geq 0.  
\end{align*}
Our results will be proven in the context of the supersonic radially expanding wave. To that end, we first prove the invariance of the supersonic expanding condition under the classical flow. 

\begin{lemma}\label{Thm:SupersonicInvariance}
    Let $u_0(r),h_0(r)\in C^1([a,b])$, where $0<a<b$, satisfy the following inequalities, 

    \begin{align*}
        &u_0(r)\geq \frac{2}{\gamma-1}h_0(r), 
        \\
        &h_0(r)=\frac{\gamma-1}{4}(w_0(r)-z_0(r))>0 . 
    \end{align*} 
    Further suppose that their time evolution, $u(t,r),h(t,r)$, remain classical over 
    $\Omega(T,[a,b])$. Then, the following holds for all $(t,r)\in\Omega(T,[a,b])$, 

    \begin{align*}
        u(t,r)\geq \frac{2}{\gamma-1}h(t,r). 
    \end{align*}
\end{lemma} 

\begin{proof}
    Suppose by way of contradiction that there exists some $(t_1,r_1)\in\Omega(T,[a,b])$ such that $u(t_1,r_1)<\frac{2}{\gamma-1}h(t_1,r_1)$. Tracing our characteristics backward, this implies there exists some $t^*,y^*$ such that the following hold,  

    \begin{align*}
        &z\circ\psi(t^*,y^*)=0 ,
        \\
        &\left.\partial_t(z\circ\psi(t,y^*))
        \right|_{t=t^*} \leq 0 .
    \end{align*} 
    However, from the lower bound on our density, we observe the following, 

    \begin{align*}
        \left.\partial_t(z\circ\psi(t,y^*))
        \right|_{t=t^*}&=\left(\frac{\gamma-1}{4r} 
        (w^2-z^2)\right)\circ\psi(t^*,y^*)
        \\
        &= \left(\frac{\gamma-1}{4r}w^2 
        \right)\circ\psi(t^*,y^*)
        \\
        &=\left(\frac{\gamma-1}{4r}\left( 
        z+\frac{4}{\gamma-1}h\right)^2 
        \right)\circ\psi(t^*,y^*)
        \\
        &=\left(\frac{4h^2}
        {(\gamma-1)r}\right)\circ\psi(t^*,y^*)>0.  
    \end{align*} 
    This implies a contradiction. 
\end{proof} 

We then seek to bound our variable coefficients, $A_i,B_i$. To that end, we first obtain an upper bound on $w$. 

\begin{lemma}\label{Thm:wUpperBound}
    Let $w_0(r),z_0(r)\in C^1([a,b])$, where $0<a<b$, satisfy the following inequalities, 

    \begin{align*}
        &z_0(r) \geq 0, 
        \\
        &h_0(r)>0 . 
    \end{align*} 
    Further suppose that their time evolution, $w(t,r),z(t,r)$, remain classical over 
    $\Omega(T,[a,b])$. Then, the following holds for all $(t,r)\in\Omega(T,[a,b])$, 

    \begin{align*}
        \lVert w(t,\cdot)\rVert_{L^\infty
        (\Omega_t([a,b]))} 
        \leq \lVert w_0\rVert_{L^\infty([a,b])} . 
    \end{align*}
\end{lemma} 

\begin{proof}
    Using Lemma (\ref{Thm:SupersonicInvariance}), this is readily observed from the time evolution of $w$, 

    \begin{align*}
        \partial_t(w\circ\eta)=\left(-\frac{\gamma-1}{4r}(w^2-z^2)\right)\circ\eta < 0, 
    \end{align*} 
    which holds for all $(t,r)\in\Omega(T,[a,b])$. 
\end{proof} 

It thus becomes clear from the upper and lower bounds in Lemmas \ref{Thm:SupersonicInvariance}, \ref{Thm:wUpperBound} that the coefficients, $A_i,B_i$ are bounded uniformly in time. This is crucial for determining the classical nature of the solution. 

\begin{theorem}[Radial Rarefactive Existence; Theorem 5.3 of \cite{CEHS_1}]\label{Thm:Radial_Rarefactive_Existence}

    Let $1<\gamma< 3$, and let the initial data, 
    $(h_0(r),u_0(r))\in C^1([0,\infty))$ satisfy 

    \begin{align*}
        0<\frac{2}{\gamma-1}h_0(r)\leq u_0(r)\leq M_1 
    \end{align*} 
    over any $(b,\infty)$ for any $b>0$. Further suppose that at zero, the solution satisfies the following, 

    \begin{align*}
        &\rho_0(0)=u_0(0)=0, 
    \end{align*}
    and that the gradient variables, $\alpha_0,\beta_0$ satisfy 

    \begin{align*}
        &\min_{r\in[0,\infty)}\{\alpha_0(r), 
        \beta_0(r)\}\geq 0, 
        \\
        &\max_{r\in[0,\infty)}\{\alpha_0(r), 
        \beta_0(r)\}\leq M_2, 
    \end{align*} 
    where $\alpha_0(0),\beta_0(0)$ are determined in the sense of their respective limits as $r\rightarrow0$. Then, the Euler Equations under radial symmetry admit a global $C^1$ solution satisfying, 

    \begin{align*}
        &\frac{2}{\gamma-1}h(r,t)\leq 
        u(r,t) \leq 2M_1 \ \forall (t,r) 
        \in [0,\infty)\times[0,\infty),  
        \\
        &h(r,t)>0 \ \forall r>0, 
        \\
        &\min_{t\geq0,r\geq0}\{\alpha(t,r), 
        \beta(t,r)\}\geq 0, 
        \\
        &\max_{t\geq0,r\geq0}\{\alpha(t,r), 
        \beta(t,r)\}\leq M_2.  
    \end{align*} 
    
\end{theorem}

\begin{remark}
    From standard theory for quasilinear symmetric-hyperbolic systems, it is known that higher regularity of the initial data persists so long as the $C_b^1$ norm of the solution remains bounded. Thus, if initial data in $C_b^m$ satisfy the conditions of Theorem \ref{Thm:Radial_Rarefactive_Existence} together with additional compatibility conditions at the origin (detailed later in Assumption \ref{Ass:A1}), then the corresponding classical solution retains this regularity on every finite time interval.
\end{remark}
While useful for proving the existence of a global classical solution, the bounds on $\alpha(t,r),\beta(t,r)$ are not strict. Indeed, \cite{ACEL_1} demonstrated that, numerically, $\alpha(t,r),\beta(t,r)$ decay in time. In anticipation of later sections in which this will become relevant, we seek to refine our bounds. 

\begin{lemma}\label{Thm:Refined_Upper_Bound}
    Let $h_0(r),u_0(r),\alpha_0(r),\beta_0(r)$ satisfy the assumptions in Theorem 
    \ref{Thm:Radial_Rarefactive_Existence}. Then, the following bound holds for all $t\geq 0$, 

    \begin{align*}
        0\leq \alpha(t,r),\beta(t,r)\leq 
        \frac{M_2}{1+\frac{\gamma+1}{8}M_2t}.
    \end{align*}
\end{lemma} 

\begin{proof}
    We first define the following, 

    \begin{align*}
        Q(t):=\frac{M_2}{1+\frac{\gamma+1}{8}M_2t}. 
    \end{align*}
    Suppose by way of contradiction that there exists some $(t_1,r_1)$ such that 
    
    \begin{align*}
        \max\{\alpha(t_1,r_1),\beta(t_1,r_1)\}
        >Q(t_1). 
    \end{align*} 
    Let $\Gamma(t_1,r_1)$ denote some closed characteristic triangle with $(t_1,r_1)$ as its top vertex. Since the initial data satisfy 
    \begin{align}
        \alpha_0(r),\beta_0(r)\leq M_2=Q(0), 
    \end{align} 
    whereas at least one of the gradient variables exceeds $Q$ at $(t_1,r_1)$, the compactness of the characteristic triangle, $\Gamma(t_1,r_1)$ implies the existence of a first upward crossing $(t^*,r^*)\in\Gamma(t_1,r_1)$. Without loss of generality, suppose that $\alpha$ crosses first.  

    \begin{align}
        &\alpha(t^*,r^*)=Q(t^*), \nonumber\\
        &\beta(t^*,r^*)\leq Q(t^*), \label{Eqn:Contradictive_Assumption}\\ 
        &\left.\partial_2(\alpha-Q)
        \right|_{(t^*,r^*)}\geq0. \nonumber
    \end{align} 
    We then evaluate the derivative of $\alpha$ along the $2$-characteristics, 

    \begin{align*}
        \partial_2\alpha(t^*,r^*) 
        &=\left.\left(-\frac{\gamma+1}{4} 
        \alpha^2-\frac{3-\gamma}{4}\alpha\beta
        +A_2\beta-B_2\alpha\right)\right|
        _{(t^*,r^*)} 
        \nonumber\\
        &\leq \left.\left(-\frac{\gamma+1}{4}Q^2 
        -\frac{3-\gamma}{4}Q\beta 
        -A_2(Q-\beta) 
        -(B_2-A_2)Q\right)\right|_{(t^*,r^*)}
        \nonumber\\
        &\leq -\frac{\gamma+1}{4}Q(t^*)^2. 
    \end{align*} 
    This allows us to compute the derivative of the difference between $\alpha$ and $Q(t)$, 

    \begin{align*}
        \left.\partial_2\left(\alpha-Q(t)\right) 
        \right|_{(t^*,r^*)} 
        &\leq -\frac{\gamma+1}{4}Q(t^*)^2 
        -(\left.\partial_tQ(t))\right|_{t^*} 
        \nonumber\\
        &=-\frac{\gamma+1}{4}Q(t^*)^2 
        +\frac{\gamma+1}{8}Q(t^*)^2 
        \nonumber\\
        &=-\frac{\gamma+1}{8}Q(t^*)^2 < 0. 
    \end{align*} 
    This contradicts (\ref{Eqn:Contradictive_Assumption}), thus completing the proof. 
\end{proof} 
Using Lemma \ref{Thm:Refined_Upper_Bound}, we obtain bounds on $w_r$, $z_r$. 

\begin{corollary}\label{Cor:FirstDerivativeDecay}
    Let $h_0(r),u_0(r),\alpha_0(r),\beta_0(r)$ satisfy the assumptions in Theorem 
    \ref{Thm:Radial_Rarefactive_Existence}. Then, the derivatives of the Riemann variables, $w_r$, $z_r$, decay according to  

    \begin{align*}
        \max\{\lVert w_r(t,\cdot)\rVert_{L^\infty}
        ,\lVert z_r(t,\cdot)\rVert_{L^\infty}\} 
        =O(1+t)^{-1}.
    \end{align*}
\end{corollary} 

\begin{proof}
    We first note that 
    \begin{align}
        &w_r=\alpha-\frac{2uh}{rc_2},  
        \label{Eqn:w_rCorr}
        \\
        &z_r=\beta+\frac{2uh}{rc_1}. 
        \label{Eqn:z_rCorr}
    \end{align} 
    from the invariance of the supersonic expanding condition under the classical flow, we note that 
    \begin{align}
        0\leq \frac{2}{\gamma-1}h\leq u.  
    \end{align} 
    Therefore, the latter term in (\ref{Eqn:w_rCorr}) is negative, allowing us to bound $w(t,r)$ above. Indeed, 
    \begin{align}
        &w(t,r)= \int_0^r w_r(t,x)dx 
        \leq rQ(t), 
        \\
        &\frac{w(t,r)}{r}\leq Q(t),   
        \label{Eqn:wUpperCorr}
    \end{align} 
    where $Q(t)$ is defined as in the proof of Lemma
    \ref{Thm:Refined_Upper_Bound}. Using this, we obtain 
    \begin{align}
        \frac{h}{r}\leq \frac{\gamma-1}{4r} 
        (w-z)\leq \frac{w}{2r}\leq\frac{1}{2}Q(t). 
        \label{Eqn:hr}
    \end{align} 
    It remains to bound the ratios $u/c_1$, $u/c_2$. The latter is trivial owing to the supersonic flow, 
    \begin{align}
        \frac{u}{c_2}=\frac{u}{u+h} 
        \leq  \frac{u}{u}=1, 
        \label{Eqn:uc2}
    \end{align} 
    whereas the former is bounded owing to the fact that $\gamma<3$. Indeed, 
    \begin{align}
        &c_1=u-h\geq \frac{3-\gamma}{2}u, 
        \\
        &\frac{u}{c_1}\leq\frac{2}{3-\gamma}. 
        \label{Eqn:uc1}
    \end{align} 
    Thus, combining (\ref{Eqn:wUpperCorr})--(\ref{Eqn:uc1}) and Lemma \ref{Thm:Refined_Upper_Bound} yields the desired decay 
    \begin{align*}
        \max\{\lVert w_r(t,\cdot)\rVert_{L^\infty}
        ,\lVert z_r(t,\cdot)\rVert_{L^\infty}\} 
        =O(1+t)^{-1}.
    \end{align*} 
\end{proof}

\subsection{Statement of Results} 

Before proceeding, we state assumptions over which our results will take place. Let $N,k$ be positive integers satisfying 
\begin{align*}
    N&>k+\frac{3}{2},  
    \\
    N&\geq 4. 
\end{align*}
We define the rarefactive background data as $(\bar h,\bar\bfu)$ with zero entropy and the perturbation as $(\sigma,\bfv,\theta)$, where 

\begin{align*}
    &H:=\bar h+\sigma , 
    \\
    &\bfU:=\bar\bfu+\bfv, 
    \\
    &S:=\bar S+\theta=\theta. 
\end{align*}

\begin{assumption}\label{Ass:A1}
    We assume the initial data of the non-perturbed solution are radially symmetric and vanish at the origin, 

    \begin{align*}
        &\bar h_0(x)\in C_b^{N+2}
        (\mathbb{R}^3), 
        \\ 
        &\bar\bfu_0(x)\in C_b^{N+2}
        (\mathbb{R}^3;\mathbb{R}^3), 
        \\
        &\bar\bfu_0(x)=\bar\bfu_0(|x|) 
        := \bar u_0(r)e_r, 
        \\
        &\bar h_0(x)=\bar h_0(|x|):=\bar h_0(r). 
    \end{align*} 
    We further assume that the following bounds hold, 

    \begin{align*}
        &0<\frac{2}{\gamma-1}\bar h_0(r)
        \leq \bar u_0(r)\leq M_1 \qquad \forall r>0, 
        \\
        &\min_{r\in[0,\infty)}\{\alpha_0(r), 
        \beta_0(r)\}\geq 0, 
        \\
        &\max_{r\in[0,\infty)}\{\alpha_0(r), 
        \beta_0(r)\}\leq M_2, 
        \\
        &\lVert\bar\bfu_0\rVert_{C_b^{N+2}}
        +\frac{2}{\gamma-1}
        \lVert\bar h_0\rVert_{C_b^{N+2}}
        \leq M_3, 
        \\
        &\left.\frac{d^j}{dr^j}\bar u_0(r)
        \right|_{r=0}=\left.\frac{d^j}{dr^j} 
        \bar h_0(r) \right|_{r=0} = 0 
        \qquad \forall \ 0\leq j\leq N+1, 
    \end{align*} 
    and define $C_0> M_1+M_2+M_3$. By Corollary \ref{Cor:FirstDerivativeDecay}, we take $C_0$ large enough so that the following bounds hold, 

    \begin{align*}
        \left\lVert w_r(t,\cdot) 
        \right\rVert_{L^\infty}+ 
        \left\lVert z_r(t,\cdot) 
        \right\rVert_{L^\infty} 
        \leq \frac{C_0}{1+t}. 
    \end{align*}
    
\end{assumption} 

\begin{assumption}\label{Ass:A2}
    The initial data for the perturbations, $(\sigma,\bfv,\theta)$, satisfy the following, 

    \begin{align*}
        &\sigma_0(x) \in H^N(\mathbb{R}^3),
        \\
        &\bfv_0(x)\in H^N(\mathbb{R}^3; 
        \mathbb{R}^3), 
        \\
        &\theta_0(x) \in H^N(\mathbb{R}^3),
        \\
        &\lVert\sigma_0\rVert_{H^N} 
        +\lVert\bfv_0\rVert_{H^N} 
        +\lVert \theta_0\rVert_{H^N}\leq \varepsilon, 
    \end{align*}  
    where the sonic speed of initial data to the perturbed system is strictly positive everywhere except possibly at the origin,  

    \begin{align*}
        &H_0(0)\geq 0, 
        \\
        &H_0(x)> 0 \qquad \forall x\neq0. 
    \end{align*}
    
\end{assumption}

We then state our main results. 

\begin{theorem}[Main Theorem]\label{Thm:PerturbedSoln}
    Let $\bar h_0, \bar\bfu_0$ satisfy Assumption 
    \ref{Ass:A1}, and let the initial perturbations, $(\sigma_0,\bfv_0, \theta_0)$, satisfy Assumption \ref{Ass:A2}. Furthermore, let 

    \begin{align*}
        &N>k+\frac{3}{2},  
        \\
        &N\geq 4, 
    \end{align*} 
    and define $T^\varepsilon_{\mathrm{life}} $ as the maximal classical lifespan of the perturbed solution. Then, for any $\delta>0$, there exists some $\varepsilon_0$ depending on $\delta$ such that $\varepsilon< \varepsilon_0$ implies that the perturbed solution both remains classical and 

    \begin{align*}
        \lVert \bfU-\bar\bfu \rVert_{W^{k,\infty}} 
        +\lVert H-\bar h\rVert_{W^{k,\infty}} 
        +\lVert \theta \rVert_{W^{k,\infty}}
        \leq\delta, 
    \end{align*}
    for all $0\leq t\leq T^\varepsilon$, where 

    \begin{align*}
        T^\varepsilon_{\mathrm{life}} \geq T^\varepsilon \gtrsim \left(\log\left(\frac{1}
        {\varepsilon}\right)\right)^{\mathcal{C}_N}, 
    \end{align*} 
    for some constant, $\mathcal{C}_N>0$. 
    
\end{theorem}

\subsection{Derivative Bounds} 

Corollary \ref{Cor:FirstDerivativeDecay} gives a refined upper bound on the first derivatives. We now turn to higher-order derivatives. Throughout this section, empty sums are understood to vanish. 

\subsubsection{Propagation of the Stationary Point-Vacuum} 

We first prove that the vanishing conditions in Assumption \ref{Ass:A1} are propagated in time. We state this result in a slightly more general form independent of symmetry assumptions. 

\begin{lemma}
    
    Let $(H_0, \bfU_0, S_0)$ satisfy  
    
    \begin{align*}
        &\partial_x^\alpha H_0(0) 
        =\partial_x^\alpha \bfU_0(0)
        =0 \ \forall |\alpha|\leq N+1, 
        \\
        &H_0(x) >0 \ \forall x\neq 0. 
    \end{align*}
    Assume that the classical solution with initial data $(H_0,\bfU_0,S_0)$ exists over $[0,T]$ with   

    \begin{align*}
    H,\mathbf U,S
    \in C([0,T];C_b^{N+2}(\mathbb R^3))
    \cap C^1([0,T];C_b^{N+1}(\mathbb R^3)).
    \end{align*}
    Then, for all $0\leq t\leq T$, 

    \begin{align*}
        &\partial_x^\alpha H(t,0) 
        =\partial_x^\alpha \bfU(t,0) 
        =0 \qquad \forall |\alpha|\leq N+1, 
        \\
        &H(t,x) >0 \qquad \forall x \neq 0. 
    \end{align*}
\end{lemma} 

\begin{proof} 
    We first define the ideal, 
    
    \begin{align*}
        \mathfrak{I}_{N+1}:= 
        \{f\in C^{N+1}:\partial_x^\alpha f(0) 
        =0 \ \forall |\alpha|\leq N+1 \}. 
    \end{align*} 
    Given that 

    \begin{align*}
        &H \in \mathfrak{I}_{N+1}(\mathbb{R}^3), 
        \\
        &\bfU \in \mathfrak{I}_{N+1}(\mathbb{R}^3; 
        \mathbb{R}^3),   
    \end{align*} 
    then we observe the following, 

    \begin{align}
        \bfU\cdot\nabla H&\in\mathfrak{I}_{N+1} 
        (\mathbb{R}^3), \label{Eqn:UGradH}
        \\
        H\nabla\cdot\bfU&\in\mathfrak{I}_{N+1} 
        (\mathbb{R}^3), \label{Eqn:HGradU}
        \\
        \bfU\cdot\nabla\bfU&\in\mathfrak{I}_{N+1} 
        (\mathbb{R}^3;\mathbb{R}^3), \label{Eqn:UGradU}
        \\
        H\nabla H&\in\mathfrak{I}_{N+1} 
        (\mathbb{R}^3;\mathbb{R}^3), \label{Eqn:HGradH} 
        \\
        H^2\nabla S&\in\mathfrak{I}_{N+1} 
        (\mathbb{R}^3;\mathbb{R}^3), \label{Eqn:H2GradS}
    \end{align} 
    thus implying that 

    \begin{align}
        &\partial_tH\in\mathfrak{I}_{N+1} 
        (\mathbb{R}^3), \label{Eqn:H_ideal}
        \\
        &\partial_t\bfU\in\mathfrak{I}_{N+1} 
        (\mathbb{R}^3;\mathbb{R}^3).  
        \label{Eqn:U_ideal}
    \end{align} 
    We then define a vector of the form, 

    \begin{align*}
        &X(t):= \left[ \partial_x^\alpha 
        H(t,0), \partial_x^\alpha\bfU(t,0)\right], 
    \end{align*} 
    where $\alpha$ ranges over $|\alpha|\leq N+1$. Differentiating with respect to time gives us an ODE, 

    \begin{align*}
        \frac{d}{dt}X(t)=F(t,X(t)), 
    \end{align*} 
    By Equations (\ref{Eqn:H_ideal})--(\ref{Eqn:U_ideal}), this ODE satisfies
    $F(t,0)=0$. Since $X(0)=0$, uniqueness for this finite-dimensional ODE
    implies that $X(t)\equiv0$ for all $0\leq t\leq T$.

    It then remains to prove that $H(t,x)>0$ for all $x\neq 0$. Indeed, we consider the flow, 

    \begin{align*}
        &\partial_t \Phi =\bfU\circ\Phi, 
        \\
        &\Phi(0,y) = y. 
    \end{align*} 
    Since $\bfU(t,0)=0$, the flow satisfies $\Phi(t,0)=0$. This gives us the time evolution of our sonic speed, 

    \begin{align*}
        &\partial_t(H\circ\Phi) 
        = \left(-\frac{\gamma-1}{2}H 
        \nabla \cdot \bfU\right)\circ\Phi, 
        \\
        &H\circ\Phi = H_0(y) 
        \exp\left(-\frac{\gamma-1}{2} 
        \int_0^t \nabla\cdot \bfU \circ\Phi ds\right) 
        > 0 \qquad \forall y\neq 0, 
    \end{align*} 
    thus implying that the sonic speed, $H$, is positive for all $y\neq 0$. Furthermore, because the solution is classical, the flow map is invertible. Since
    $\Phi(t,0)=0$, every $x\neq0$ has a preimage $y\neq0$. Therefore, $H(t,x)>0$ for all $x\neq0$.
\end{proof} 

\subsubsection{Higher-Order Characteristic Identities}

We seek to bound the time evolution of the radial derivatives of $w,z$. Since $\partial_1,\partial_2$ do not commute with $\partial_r$, we begin with the commutator relation. Differentiating (\ref{Eqn:partial_1})--(\ref{Eqn:partial_2}) gives, for $i=1,2$,  

\begin{align*}
    \partial_r^j(\partial_t+c_i\partial_r) 
    &=\partial_t\partial_r^j 
    +\sum_{k=0}^j \binom{j}{k}(\partial^{j-k}_rc_i) 
    \partial_r^{k+1} 
    \nonumber\\
    &=\partial_t\partial_r^j+c_i\partial_r^{j+1} 
    +\sum_{k=0}^{j-1}\binom{j}{k}(\partial_r^{j-k} 
    c_i)\partial_r^{k+1} .
\end{align*} 
Rearranging allows us to obtain the desired evolution of our higher order derivatives, 

\begin{align}
    &\partial_t\partial_r^j+c_i\partial_r^{j+1} 
    =\partial_r^j(\partial_t+c_i\partial_r) 
    -\sum_{k=0}^{j-1}\binom{j}{k}
    (\partial_r^{j-k}c_i)\partial_r^{k+1}, 
    \nonumber\\
    &\partial_i\partial_r^j=\partial_r^j\partial_i 
    -\sum_{k=0}^{j-1}\binom{j}{k}(\partial_r^{j-k} 
    c_i)\partial_r^{k+1} . 
    \label{Eqn:Commutator_Relation}
\end{align} 
This commutator identity is the starting point for the higher-order derivative estimates of our Riemann variables. For $j\geq2$, the calculations in Appendix \ref{App:CharacteristicIdentities} give  

\begin{align*}
    \partial_1(\partial_r^jz)=-I^z_j\partial_r^jz 
     +\Lambda^z_j+Q^z_j+S^z_j , 
\end{align*} 
where 

\begin{align}
    &S_j^z:=-\sum_{k=1}^{j-2}
    \binom{j}{k}\left(\frac{\gamma-1}
    {4}\left(r^{-1}\partial_r^kz\partial_r^{j-k}z 
    +\partial_r^k(z^2)\partial_r^{j-k}(r^{-1})\right) 
    +\partial_r^{k+1}z\partial_r^{j-k}c_1\right), 
    \label{Eqn:S_j^z}\\
    &Q_j^z:=-\frac{\gamma-1}{4}
    \left(jr^{-1}z_r\partial_r^{j-1}z
    +z^2\partial_r^j(r^{-1}) 
    -jr^{-2}\partial_r^{j-1}(z^2) 
    \right)\nonumber\\
    &+\frac{\gamma-1}{4}\left(r^{-1}\sum_{k=1}^{j-1}\binom{j}{k}\partial_r^{j-k}w 
    \partial_r^kw+\sum_{k=0}^{j-1}\binom{j}{k} 
    \partial_r^{j-k}(r^{-1})\partial_r^k(w^2)\right),
    \label{Eqn:q_j^z}\\
    &\Lambda^z_j:=\left(\frac{\gamma-1}{2r}w
    -\frac{3-\gamma}{4}z_r\right)\partial_r^jw , 
    \label{Eqn:Lambda_j^z}\\
    &I^z_j:=\frac{\gamma-1}{2r}z+\frac{\gamma+1}{4}
    (j+1)z_r+\frac{3-\gamma}{4}jw_r . 
    \label{Eqn:I_j^z}
\end{align} 
The time evolution of $z$ along $\psi$ is then  

\begin{align}
    &\partial_t\left(\partial_r^jz\circ\psi\right)= 
    \left(-I^z_j\partial_r^jz  
    +\Lambda^z_j+Q^z_j+S^z_j  \right)\circ\psi, 
    \label{Eqn:jzEvolution}
\end{align}
which can be solved explicitly using Duhamel's formula, 

\begin{align}
    \partial_r^jz\circ\psi
    &=
    \partial_r^jz_0
    \exp\left(-\int_0^t I^z_j\circ\psi \,d\tau \right)
    \nonumber\\
    &\quad
    +\int_0^t
    \exp\left(-\int_s^t I^z_j\circ\psi \,d\tau\right)
    \left(\Lambda^z_j+Q^z_j+S^z_j\right)\circ\psi \,ds, 
    \label{Eqn:jzDuhamel}
\end{align}
where here and below, $w_0,z_0$ are understood as being evaluated at the characteristic footpoints. Applying the same calculation gives the corresponding identity for $w$: 

\begin{align*}
    \partial_2(\partial_r^jw)=-I^w_j\partial_r^jw 
     +\Lambda^w_j+Q^w_j+S^w_j , 
\end{align*} 
where 

\begin{align}
    &S_j^w:=-\sum_{k=1}^{j-2}
    \binom{j}{k}\left(\frac{\gamma-1}
    {4}\left(r^{-1}\partial_r^kw\partial_r^{j-k}w 
    +\partial_r^k(w^2)\partial_r^{j-k}(r^{-1})\right) 
    +\partial_r^{k+1}w\partial_r^{j-k}c_2\right), 
    \label{Eqn:S_j^w}\\
    &Q_j^w:=-\frac{\gamma-1}{4}
    \left(jr^{-1}w_r\partial_r^{j-1}w
    +w^2\partial_r^j(r^{-1})
    -jr^{-2}\partial_r^{j-1}(w^2) 
    \right) 
    \nonumber\\
    &+\frac{\gamma-1}{4}\left(r^{-1}\sum_{k=1}^{j-1}\binom{j}{k}\partial_r^{j-k}z 
    \partial_r^kz+\sum_{k=0}^{j-1}\binom{j}{k} 
    \partial_r^{j-k}(r^{-1})\partial_r^k(z^2)\right),
    \label{Eqn:q_j^w}\\
    &\Lambda^w_j:=\left(\frac{\gamma-1}{2r}z
    -\frac{3-\gamma}{4}w_r\right)\partial_r^jz, 
    \label{Eqn:Lambda_j^w}\\
    &I^w_j:=\frac{\gamma-1}{2r}w+\frac{\gamma+1}{4}
    (j+1)w_r+\frac{3-\gamma}{4}jz_r. 
    \label{Eqn:I_j^w}
\end{align} 
The time evolution of $w$ along $\eta$ is therefore 

\begin{align}
    &\partial_t\left(\partial_r^jw\circ\eta\right)= 
    \left(-I^w_j\partial_r^jw  
    +\Lambda^w_j+Q^w_j+S^w_j \right)\circ\eta. 
    \label{Eqn:jwEvolution}
\end{align} 
Solving this equation explicitly gives 

\begin{align}
    \partial_r^jw\circ\eta
    &=
    \partial_r^jw_0
    \exp\left(-\int_0^t I^w_j\circ\eta \,d\tau \right)
    \nonumber\\
    &\quad
    +\int_0^t
    \exp\left(-\int_s^t I^w_j\circ\eta \,d\tau\right)
    \left(\Lambda^w_j+Q^w_j+S^w_j\right)\circ\eta \,ds .
    \label{Eqn:jwDuhamel}
\end{align}
We now estimate the terms defined in \eqref{Eqn:S_j^z}--\eqref{Eqn:I_j^z} and \eqref{Eqn:S_j^w}--\eqref{Eqn:I_j^w}. The case $j=2$ is treated separately, while the estimates for $j\geq3$ are obtained inductively.

\subsubsection{Bounding $I_j$} 

We first note that the vanishing conditions in Assumption \ref{Ass:A1} are propagated in time by the radial background flow. Therefore, the following Taylor estimates hold for all $t\geq0$, 

\begin{align}
    &\left\lVert \frac{\partial_r^mw}{r^n} 
    \right\rVert_{L^\infty}\leq 
    \frac{1}{n!}\lVert\partial_r^{m+n} w\rVert
    _{L^\infty}, \label{Eqn:wOverr}\\
    &\left\lVert \frac{\partial_r^mz}{r^n} 
    \right\rVert_{L^\infty}\leq 
    \frac{1}{n!}\lVert\partial_r^{m+n} z\rVert
    _{L^\infty} \label{Eqn:zOverr}, 
\end{align} 
for all $m+n\leq N+1$. For $n=1$, the first derivative bounds in Assumption \ref{Ass:A1} imply that 

\begin{align*}
    &\left\lVert \frac{w}{r} 
    \right\rVert_{L^\infty}\leq 
    \lVert w_r\rVert
    _{L^\infty},\\
    &\left\lVert \frac{z}{r} 
    \right\rVert_{L^\infty}\leq 
    \lVert z_r\rVert
    _{L^\infty} .  
\end{align*} 
Inserting these into (\ref{Eqn:I_j^z}), (\ref{Eqn:I_j^w}) yields 

\begin{align*}
    |I_j^w(t,r)|,|I_j^z(t,r)| \leq \frac{2C_0(j+1)}
    {1+t}. 
\end{align*} 
For simplicity, we denote 

\begin{align*}
    I_j(t):=\max\left\{\lVert 
    I_j^w(t,\cdot)\rVert_{L^\infty},
    \left\lVert I_j^z(t,\cdot)
    \right\rVert_{L^\infty}
    \right\}, 
\end{align*}
which gives us  

\begin{align}
    &\int_s^t I_j(\tau) \, d\tau \leq 
    2C_0(j+1)\log\left(\frac{1+t}{1+s} \right), 
    \nonumber\\
    &\exp\left(\int_s^t I_j(\tau) \, d\tau \right) 
    \leq \left(\frac{1+t}{1+s} \right) 
    ^{2C_0(j+1)} , 
    \label{Eqn:I1_Int}\\ 
    &\exp\left(\int_0^t I_j(\tau) \, d\tau \right) 
    \leq \left(1+t\right) 
    ^{2C_0(j+1)}. \nonumber 
\end{align} 

\subsubsection{Bounding $\Lambda_j$} 

We then move to bound the coefficients of the highest-order derivatives in $\Lambda_j^z$, $\Lambda_j^w$. Using (\ref{Eqn:wOverr})--(\ref{Eqn:zOverr}) along with the first derivative bound in Assumption \ref{Ass:A1} implies   

\begin{align}
    \left\lvert \frac{\gamma-1}{2r}z
    -\frac{3-\gamma}{4}w_r \right\rvert 
    +\left\lvert \frac{\gamma-1}{2r}w
    -\frac{3-\gamma}{4}z_r \right\rvert 
    \leq \frac{4C_0}{1+t}. 
    \label{Eqn:Lambda_Coeff_Bound}
\end{align} 

\subsubsection{Bounding $Q_j$} 

We now estimate $Q_j^w,Q_j^z$. To that end, we define new terms,  

\begin{align*}
    &Z_j(t):=\sup_{r\geq0}\left\lvert 
    \partial_r^jz(t,r)\right\rvert,  
    \\
    &W_j(t):=\sup_{r\geq0}\left\lvert 
    \partial_r^jw(t,r)\right\rvert, 
    \\
    &\mathcal{M}_j(t):=\max\left\{
    Z_j(t),W_j(t)\right\} , 
    \\
    &Q_j(t):=\max\left\{\lVert 
    Q_j^w(t,\cdot)\rVert_{L^\infty},
    \left\lVert Q_j^z(t,\cdot)
    \right\rVert_{L^\infty}
    \right\}. 
\end{align*} 
Furthermore, for $j\geq2$, let 

\begin{align*}
    &\mathcal{L}_j(t):=\max_{2\leq m\leq j} 
    \mathcal{M}_m(t) . 
\end{align*} 
For $j\geq3$, the calculations in Appendix \ref{App:QjBound} give 

\begin{align*}
    Q_j(t)\leq D_j^1\mathcal{L}_{j-1}(t)^2, 
\end{align*} 
where $D_j^1$ is some constant depending only on $j$.

\subsubsection{Bounding $S_j$} 

We then seek to estimate $S_j^w$, $S_j^z$. We define 

\begin{align*}
    S_j(t):=\max\left\{\lVert 
    S_j^w(t,\cdot)\rVert_{L^\infty},
    \left\lVert S_j^z(t,\cdot)
    \right\rVert_{L^\infty}
    \right\}. 
\end{align*}
For $j\geq3$, the calculations in Appendix \ref{App:SjBound} give 

\begin{align*}
    S_j(t)\leq D_j^2 \mathcal{L}_{j-1}(t)^2, 
\end{align*} 
where $D_j^2$ is some constant depending only on $j$. We next establish the base case $j=2$ for the higher-order induction. 

\subsection{Second Derivative Bounds} 

We compute explicitly 

\begin{align*}
    &S_2^z=0, 
    \\
    &Q_2^z=\frac{\gamma-1}{2r^3}\left( 
    r^2(w_r^2-z_r^2)+2r(zz_r-ww_r)+(w^2-z^2)\right), 
    \\
    &\Lambda_2^z=\left(\frac{\gamma-1}{2r}w
    -\frac{3-\gamma}{4}z_r\right)w_{rr}, 
    \\
    &S_2^w=0 , 
    \\
    &Q_2^w=\frac{\gamma-1}{2r^3}\left( 
    r^2(z_r^2-w_r^2)+2r(ww_r-zz_r)+(z^2-w^2)\right),  
    \\
    &\Lambda_2^w=\left(\frac{\gamma-1}{2r}z
    -\frac{3-\gamma}{4}w_r\right)z_{rr}. 
\end{align*} 
We further note the bounds, 

\begin{align*}
    &\left\lvert \frac{w}{r} \right\rvert 
    ,\left\lvert \frac{z}{r} \right\rvert 
    \leq \frac{C_0}{1+t}, 
    \\
    &\left\lvert\frac{w_r}{r}\right\rvert, 
    \left\lvert \frac{2w}{r^2} \right\rvert, 
    \left\lvert\frac{z_r}{r}\right\rvert, 
    \left\lvert \frac{2z}{r^2} \right\rvert\leq 
    \mathcal{M}_2(t), 
\end{align*} 
which give  

\begin{align}
    &|Q_2^z|, |Q_2^w|\leq  \frac{7C_0}
    {1+t}\mathcal{M}_2(t) 
    \label{Eqn:Q2Bounds}. 
\end{align}
It follows from (\ref{Eqn:I1_Int})--(\ref{Eqn:Q2Bounds}) that 

\begin{align*}
    &|z_{rr}\circ\psi|\leq C_0(1+t)^{6C_0}
    \left(1+ 11\int_0^t 
    \frac{\mathcal{M}_2(s)}{(1+s)^{6C_0+1}} 
    ds\right) , 
    \\
    &|w_{rr}\circ\eta|\leq C_0(1+t)^{6C_0}
    \left(1+ 11\int_0^t 
    \frac{\mathcal{M}_2(s)}{(1+s)^{6C_0+1}} 
    ds\right) . 
\end{align*} 
Taking the supremum over $r\geq0$ and then the maximum between $W_2,Z_2$ implies 

\begin{align}
    &\mathcal{M}_2(t)\leq C_0(1+t)^{6C_0}
    \left(1+ 11\int_0^t 
    \frac{\mathcal{M}_2(s)}{(1+s)^{6C_0+1}} 
    ds\right) . 
    \label{Eqn:M2_Integral_Form}
\end{align} 
We then define 

\begin{align*}
    \mathcal{N}_2(t):=\mathcal{M}_2(t)(1+t)^{-6C_0},  
\end{align*}
which transforms (\ref{Eqn:M2_Integral_Form}). 

\begin{align*}
    \mathcal{N}_2(t)\leq C_0+11C_0\int_0^t
    \frac{\mathcal{N}_2(s)}{1+s}ds . 
\end{align*} 
By Gr\"onwall's inequality, we obtain the upper bounds,  

\begin{align*}
    &\mathcal{N}_2(t)\leq C_0(1+t)^{11C_0}, 
    \\
    &\mathcal{M}_2(t)\leq C_0(1+t)^{17C_0}. 
\end{align*} 
This completes the second derivative estimate. 

\subsection{Higher Derivative Bounds} 

We now seek to inductively bound $\mathcal{M}_j(t)$ for $3\leq j\leq N+1$. Defining 
\begin{align*}
    &D_j:=D_j^1+D_j^2,  
    \\
    &a_j:=2C_0(j+1), 
\end{align*}
gives us 

\begin{align*}
    |\partial_r^jw\circ\eta|&
    \leq(1+t)^{a_j}\left(\left|
    \partial_r^jw_0 \right| 
    +\int_0^t \left(\frac{4C_0\mathcal{M}_j(s)}
    {(1+s)^{a_j+1}}+\frac{D_j
    \mathcal{L}_{j-1}(s)^2}{(1+s)^{a_j}}\right)
    ds\right), 
    \\
    |\partial_r^jz\circ\psi|&\leq 
    (1+t)^{a_j}\left(\left|
    \partial_r^jz_0 \right| 
    +\int_0^t \left(\frac{4C_0\mathcal{M}_j(s)}
    {(1+s)^{a_j+1}}+\frac{D_j
    \mathcal{L}_{j-1}(s)^2}{(1+s)^{a_j}}\right)
    ds\right). 
\end{align*} 
Taking their supremum over $r\geq0$ and then the maximum of both equations gives 

\begin{align}
    \mathcal{M}_j(t)\leq (1+t)^{a_j}\left(
    C_0+\int_0^t \left(\frac{4C_0\mathcal{M}_j(s)}
    {(1+s)^{a_j+1}}+\frac{D_j
    \mathcal{L}_{j-1}(s)^2}{(1+s)^{a_j}}\right)
    ds\right). 
    \label{Eqn:Mj_Integral_Form}
\end{align} 
As before, we define 

\begin{align*}
    \mathcal{N}_j(t):=(1+t)^{-a_j}\mathcal{M}_j(t), 
\end{align*} 
which transforms (\ref{Eqn:Mj_Integral_Form}), 

\begin{align}
    \mathcal{N}_j(t)\leq C_0+4C_0\int_0^t 
    \frac{\mathcal{N}_j(s)}{1+s}ds+D_j\int_0^t 
    \frac{\mathcal{L}_{j-1}(s)^2}{(1+s)^{a_j}}ds. 
    \label{Eqn:Nj_Int}
\end{align} 
We now proceed with the induction. The base case is already finished with the bounding of $\mathcal{M}_2(t)$. Now, assume that $j\geq 3$. Assume by way of induction that there exist constants, $A_m,B_m$ such that  

\begin{align}
    \mathcal{M}_m(t)\leq A_m(1+t)^{B_m} 
    \label{Eqn:wzInductiveBounds}
\end{align} 
holds for all $2\leq m\leq j-1$. We then define $A_{\ell}^*,B_{\ell}^*$ for $1\leq \ell\leq N+1$ as

\begin{align*}
    A_\ell^*
    &:=
    \begin{cases}
        C_0, & \ell=1,
        \\
        \displaystyle\max_{2\leq m\leq\ell}A_m,
        & 2\leq\ell\leq N+1,
    \end{cases}
    \\
    B_\ell^*
    &:=
    \begin{cases}
        0, & \ell=1,
        \\
        \displaystyle\max_{2\leq m\leq\ell}B_m,
        & 2\leq\ell\leq N+1.
    \end{cases}
\end{align*}
These definitions are used within this inductive proof and retained throughout the remainder of this paper. This gives us  

\begin{align*}
    &\mathcal{L}_{j-1}(t)\leq A_{j-1}^* 
    (1+t)^{B_{j-1}^*}, 
    \\
    &\mathcal{L}_{j-1}(t)^2\leq (A_{j-1}^*)^2 
    (1+t)^{2B_{j-1}^*} . 
\end{align*} 
We then compute 

\begin{align*}
    D_j\int_0^t 
    \frac{\mathcal{L}_{j-1}(s)^2}{(1+s)^{a_j}}ds 
    &\leq D_j(A_{j-1}^*)^2\int_0^t 
    (1+s)^{2B_{j-1}^*-a_j}ds. 
\end{align*} 
This integral is bounded above by a polynomial in $1+t$.

\begin{align}
    D_j(A_{j-1}^*)^2\int_0^t 
    (1+s)^{2B_{j-1}^*-a_j}ds\leq 
    \hat A_j (1+t)^{\hat B_j} 
    \label{Eqn:Dj_Term_Int},  
\end{align}
where $\hat A_j,\hat B_j > 0$. We then insert (\ref{Eqn:Dj_Term_Int}) into (\ref{Eqn:Nj_Int}) to obtain 

\begin{align*}
    \mathcal{N}_{j}(t)\leq C_0+
    \hat A_j (1+t)^{\hat B_j}
    +4C_0\int_0^t\frac{\mathcal{N}_j(s)}{1+s}ds,  
\end{align*}
which may be bounded by Gr\"onwall's inequality. 

\begin{align*}
    &\mathcal{N}_j(t)\leq \left( 
    C_0+\hat A_j (1+t)^{\hat B_j}\right) 
    (1+t)^{4C_0} . 
\end{align*} 
Returning to $\mathcal{M}_j(t)=(1+t)^{a_j}\mathcal{N}_j(t)$, we obtain

\begin{align*}
    \mathcal{M}_j(t)
    \leq (C_0+\hat A_j)(1+t)^{a_j+\hat B_j+4C_0}.
\end{align*}
Thus we have 

\begin{align*}
    &A_j:=C_0+\hat A_j,
    \\
    &B_j:=a_j+\hat B_j+4C_0.
\end{align*}
This completes the induction. 

\begin{remark}
    We note that the refined upper bound for the first derivatives of the rarefactive background in Lemma \ref{Thm:Refined_Upper_Bound} is crucial for obtaining polynomial upper bounds on the higher derivatives. Indeed, repeating the induction under only a uniform constant bound on the first derivatives yields exponential upper bounds for each higher derivative. Repeating the calculations in Section 4 under this more general bound yields a stability timespan of order 
    \begin{align*}
        \log\log\frac{1}{\varepsilon}. 
    \end{align*}
    Thus, it is clear that the rarefactive nature of the background substantially improves the estimated stability lifespan from doubly logarithmic to polylogarithmic order. 
\end{remark} 

\subsection{Cartesian Derivatives} 

In preparation for the next section, we seek to bound our Cartesian derivatives in terms of our radial bounds. To that end, we utilize Lemmas \ref{Lemma:CartRadScalar} and \ref{Lemma:CartRadVector} in Appendix \ref{App:RadialCartesian}. By Lemma \ref{Lemma:CartRadScalar} and Assumption \ref{Ass:A1}, 

\begin{align*}
    \left|\partial^\alpha \bar h(x)\right| 
    \lesssim \sum_{j=1}^k 
    \frac{|\partial_r^j\bar h(r)|}{r^{k-j}}, 
\end{align*} 
where $|\alpha|=k$, and $k\leq N+1$. Furthermore, since 

\begin{align*}
    \partial_r^j\bar h(t,0)=0 \qquad \forall 0\leq j\leq N+1, 
\end{align*}
we have the following bound, for some constant, $D_\alpha$,  

\begin{align*}
    \lVert\partial^\alpha \bar h\rVert
    _{L^\infty(\mathbb{R}^3)} 
    \leq D_\alpha \lVert \partial_r^k \bar h\rVert 
    _{L^\infty(\mathbb{R})} .
\end{align*} 
We further note that in the particular case of $|\alpha|=1$ that $D_\alpha=1$,  

\begin{align*}
    \lVert \nabla \bar h\rVert_{L^\infty(\mathbb{R}^3)} 
    = \lVert \partial_r\bar h\rVert_{L^\infty(\mathbb{R})}.
\end{align*} 
Similarly, by Lemma \ref{Lemma:CartRadVector} and Assumption \ref{Ass:A1},  

\begin{align*}
    \lVert\partial^\alpha\bar\bfu\rVert
    _{L^\infty}\leq D_\alpha \lVert 
    \partial_r^k \bar u\rVert_{L^\infty}, 
\end{align*} 
for any $|\alpha|=k\leq N+1$. 

\section{Evolution of the Perturbed Solution} 

We now wish to perturb our classical solution. The Euler Equations prior to singularity formation are given by, 

\begin{align}
    &h_t+\bfu\cdot\nabla h+ah\nabla\cdot\bfu=0,
    \label{Eqn:Non-Isent1}
    \\
    &\bfu_t+(\bfu\cdot\nabla)\bfu+\frac{1}{a}h\nabla h
    -\frac{1}{2a\gamma}h^2\nabla S=0,
    \label{Eqn:Non-Isent2}
    \\
    &S_t+\bfu\cdot\nabla S=0, 
    \label{Eqn:Non-Isent3}
\end{align}
where $a:=(\gamma-1)/2$. Let the background solution be denoted by 

\begin{align*}
    \mathcal{B}:=(\bar h,\bar\bfu,\bar S), 
\end{align*}
and let 

\begin{align*}
    \mathcal{W}:=(\sigma,\bfv,\theta), 
\end{align*} 
denote the perturbation. Then, define the perturbed solution as 

\begin{align*}
    (H,\bfU,S):=(\bar h+\sigma, \bar\bfu+\bfv, 
    \bar S+\theta). 
\end{align*}
Using these, we obtain the perturbed system, 

\begin{align*}
    &(\bar h+\sigma)_t+(\bar\bfu+\bfv)
    \cdot\nabla (\bar h+\sigma)
    +a(\bar h+\sigma)\nabla
    \cdot (\bar\bfu+\bfv)=0, 
    \\
    &(\bar\bfu+\bfv)_t+((\bar\bfu+\bfv)\cdot\nabla)
    (\bar\bfu+\bfv)+\frac{1}{a}(\bar h+\sigma)
    \nabla (\bar h+\sigma)-\frac{(\bar h+\sigma)^2}
    {2a\gamma}\nabla (\bar S+\theta) = 0, 
    \\
    &(\bar S+\theta)_t+(\bar\bfu+\bfv)\cdot\nabla 
    (\bar S+\theta)= 0 . 
\end{align*} 
Using the fact that $\mathcal{B}$ is a solution to (\ref{Eqn:Non-Isent1})--(\ref{Eqn:Non-Isent3}), we obtain the simplified system, 

\begin{align*}
    &\sigma_t
    +\bfU\cdot\nabla\sigma
    +\mathbf v\cdot\nabla\bar h
    +aH\nabla\cdot\mathbf v
    +a\sigma\nabla\cdot\bar\bfu
    =0, 
    \\
    &\mathbf v_t
    +\bfU\cdot\nabla\mathbf v
    +\mathbf v\cdot\nabla\bar\bfu
    +\frac{1}{a}H\nabla\sigma
    +\frac{1}{a}\sigma\nabla\bar h 
    -\frac{H^2}{2a\gamma}\nabla \theta 
    -\frac{H^2-\bar h^2}{2a\gamma}\nabla \bar S
    =0, 
    \\
    &\theta_t+\bfU\cdot\nabla \theta 
    +\bfv\cdot\nabla\bar S=0. 
\end{align*} 
We then rearrange the system to obtain

\begin{align}
    &\sigma_t
    +\bfU\cdot\nabla\sigma
    +aH\nabla\cdot\mathbf v
    =-\mathbf v\cdot\nabla\bar h
    -a\sigma\nabla\cdot\bar\bfu, 
    \label{Eqn:Perturbed1}
    \\
    &\mathbf v_t
    +(\bfU\cdot\nabla)\mathbf v
    +\frac{1}{a}H\nabla\sigma 
    -\frac{H^2}{2a\gamma}\nabla\theta 
    =-(\bfv\cdot\nabla)\bar\bfu
    -\frac{1}{a}\sigma\nabla\bar h
    +\frac{H^2-\bar h^2}{2a\gamma}\nabla\bar S, 
    \label{Eqn:Perturbed2}
    \\
    &\theta_t+\bfU\cdot\nabla \theta= 
    -\bfv\cdot\nabla\bar S. 
    \label{Eqn:Perturbed3}
\end{align} 

\subsection{Energy Form} 

We seek an appropriate quadratic form which will adequately bound the $H^N$ norm of our perturbation. To that end, we first define the energy density, 

\begin{align}
    &e_\alpha:=(\partial^\alpha\sigma)^2 
    +a^2\left|\partial^\alpha\bfv\right|^2 
    +\lambda(\partial^\alpha \theta)^2 
    -\frac{1}{\gamma}H\partial^\alpha \theta 
    \partial^\alpha\sigma, 
    \label{Eqn:ptwiseEnergyDensity}
\end{align} 
where $|\alpha|=k\leq N$ is some multi-index, and where $\lambda$ will be defined later to be large enough so as to make $e_\alpha$ coercive. Using this, we define the $k$-th order energy, 

\begin{align*}
    E_k(t):=\frac{1}{2}\int_{\mathbb{R}^3} 
    \sum_{|\alpha|=k}e_\alpha \, dx, 
\end{align*} 
and the cumulative $N$-th order energy, 

\begin{align*}
    \mathcal{E}_N(t):=\sum_{k=0}^N E_k(t). 
\end{align*} 
It remains to show that $e_\alpha$ is coercive with respect to its inputs. This will be achieved primarily through bootstrapping. 

\begin{assumption}[Bootstrap Assumption]
\label{Ass:Bootstrap}
    Assume that 

    \begin{align*}
        \lVert H \rVert_{L^\infty([0,T^\varepsilon]
        \times \mathbb{R}^3)} \leq 5C_0. 
    \end{align*}
    
\end{assumption} 

Using Assumption \ref{Ass:Bootstrap}, we define 

\begin{align*}
    \lambda:=10C_0^2, 
\end{align*} 
and claim that $e_\alpha$ is coercive over $[0,T^\varepsilon]$. Indeed, (\ref{Eqn:ptwiseEnergyDensity}) can be rewritten as 

\begin{align*}
    &e_\alpha=\left(\partial^\alpha\sigma 
    -\frac{H}{2\gamma}\partial^\alpha \theta\right)^2 
    +a^2|\partial^\alpha\bfv|^2 
    +\left(\lambda-\frac{H^2}{4\gamma^2}\right) 
    (\partial^\alpha \theta)^2.  
\end{align*} 
Then, noting that 

\begin{align}
    \lambda-\frac{H^2}{4\gamma^2} \geq 
    10C_0^2-\frac{25C_0^2}{4}>0, 
    \label{Eqn:LambdaCoercivity}
\end{align} 
establishes the coercivity of $e_\alpha$. 

\subsection{Zeroth-Order Energy} 
\label{Subsection:ZerothEnergy}

We seek to compute 

\begin{align*}
    E_0(t):=\frac{1}{2}\int_{\mathbb{R}^3} 
    \left(\sigma^2+a^2|\bfv|^2+\lambda 
    \theta^2-\frac{1}{\gamma}H\theta\sigma\right) 
    \, dx . 
\end{align*} 
We multiply (\ref{Eqn:Perturbed1}) by 

\begin{align*}
    \sigma-\frac{H}{2\gamma}\theta, 
\end{align*} 
giving us 

\begin{align}
    \frac{1}{2}&(\sigma^2)_t 
    -\frac{H}{2\gamma}\theta\sigma_t 
    +\frac{1}{2}\bfU\cdot\nabla
    \sigma^2 
    -\frac{H}{2\gamma} \theta 
    \bfU\cdot\nabla\sigma 
    +aH\sigma\nabla\cdot
    \bfv-\frac{aH^2}{2\gamma}
    \theta\nabla\cdot\bfv\nonumber\\
    &=\left(\sigma
    -\frac{H}{2\gamma}\theta\right)
    \left( 
    -\bfv\cdot\nabla\bar h 
    -a\sigma\nabla\cdot\bar\bfu
    \right). 
    \label{Eqn:sigmatEnergy}
\end{align} 
Then multiply (\ref{Eqn:Perturbed2}) by $a^2\bfv$, 

\begin{align}
    &\frac{a^2}{2}(|\bfv|^2)_t 
    +\frac{a^2}{2}\bfU\cdot\nabla|\bfv|^2 
    +aH\bfv\cdot\nabla\sigma 
    -\frac{a}{2\gamma}H^2\bfv\cdot\nabla \theta 
    \nonumber\\
    &=-a^2\bfv\cdot((\bfv\cdot\nabla)\bar\bfu) 
    -a\sigma\bfv\cdot\nabla\bar h 
    +\frac{a(H^2-\bar h^2)}{2\gamma} 
    \bfv\cdot\nabla\bar S. 
    \label{Eqn:vtEnergy}
\end{align}
Then multiplying (\ref{Eqn:Perturbed3}) by 

\begin{align*}
    \lambda \theta-\frac{H}{2\gamma}\sigma, 
\end{align*} 
gives us 

\begin{align}
    \frac{\lambda}{2}(\theta^2)_t -\frac{H}{2\gamma} 
    \sigma \theta_t + \frac{\lambda}{2}\bfU\cdot 
    \nabla(\theta^2)-\frac{H}{2\gamma}\sigma 
    \bfU\cdot\nabla \theta=-\left( 
    \lambda\theta-\frac{H\sigma}{2\gamma}\right) 
    \bfv\cdot\nabla\bar S. 
    \label{Eqn:StEnergy}
\end{align} 
We then add (\ref{Eqn:sigmatEnergy})--(\ref{Eqn:StEnergy}) to obtain 

\begin{align*}
    &\frac{1}{2}(e_0)_t+\frac{\theta\sigma}{2\gamma}H_t 
    +\frac{1}{2}\bfU\cdot\nabla e_0+\frac{\theta\sigma}
    {2\gamma}\bfU\cdot\nabla H 
    +a\nabla\cdot\left(H\left( 
    \sigma-\frac{H\theta}{2\gamma}\right)\bfv 
    \right) 
    \nonumber\\
    &+a\left(\frac{H\theta}{\gamma}-\sigma\right)\bfv 
    \cdot\nabla H=\left(\sigma-\frac{H\theta}{2\gamma} 
    \right)\left(-\bfv\cdot\nabla\bar h 
    -a\sigma\nabla\cdot\bar\bfu\right)
    \nonumber\\
    &-a^2\bfv\cdot((\bfv\cdot\nabla)\bar\bfu) 
    -a\sigma\bfv\cdot\nabla\bar h 
    +\left(\frac{\sigma((a+1)H+a\bar h)}{2\gamma}  
    -\lambda\theta  
    \right)\bfv\cdot\nabla\bar S. 
\end{align*} 
By a standard spatial cutoff argument, we obtain, 

\begin{align*}
    &\int_{\mathbb{R}^3} 
    a\nabla\cdot\left(H\left( 
    \sigma-\frac{H\theta}{2\gamma}
    \right)\bfv 
    \right) \, dx = 0, 
    \\
    &\int_{\mathbb{R}^3}\bfU\cdot\nabla e_0 
    \, dx = -\int_{\mathbb{R}^3}e_0
    \nabla\cdot\bfU \, dx. 
\end{align*}  
Furthermore, we note that 

\begin{align*}
    H_t+\bfU\cdot\nabla H=-aH\nabla\cdot\bfU. 
\end{align*}
Therefore, 

\begin{align}
    \frac{d}{dt}E_0(t)=&\frac{1}{2}\int_{\mathbb{R}^3}
    e_0\nabla\cdot\bfU\, dx+\frac{a}{2\gamma} 
    \int_{\mathbb{R}^3}H\theta\sigma\nabla\cdot\bfU 
    \,dx\nonumber\\
    -&a\int_{\mathbb{R}^3}\left(\frac{H\theta}{\gamma} 
    -\sigma\right)\bfv\cdot\nabla H\,dx 
    \nonumber\\
    -&\int_{\mathbb{R}^3}\left( 
    \sigma-\frac{H\theta}{2\gamma}\right)
    \left(\bfv\cdot\nabla\bar h 
    +a\sigma\nabla\cdot\bar\bfu\right)\,dx 
    \nonumber\\
    -&a^2\int_{\mathbb{R}^3}\bfv\cdot((\bfv 
    \cdot\nabla)\bar\bfu)\,dx 
    -a\int_{\mathbb{R}^3}\sigma\bfv\cdot 
    \nabla\bar h\,dx 
    \nonumber\\
    +&\int_{\mathbb{R}^3}\left(\frac{\sigma
    ((a+1)H+a\bar h)}{2\gamma} 
    -\lambda\theta  
    \right)\bfv\cdot\nabla\bar S \, dx
    \label{Eqn:E0_Evolution}
\end{align}

\subsection{Higher-Order Energy}

We now derive energy estimates for $1\leq k\leq N$. First, define the following notation for multi-indices, 

\begin{align*}
    \binom{\alpha}{\beta}:=\binom{\alpha_1}{\beta_1} 
    \binom{\alpha_2}{\beta_2}\binom{\alpha_3}{\beta_3},  
\end{align*} 
where $\beta\leq\alpha$ if $\beta_i\leq\alpha_i$ for all $i$. We then introduce the following commutator relations where $|\alpha|=k\geq1$, 

\begin{align*}
    &\mathcal{J}_1^\alpha:=\partial^\alpha (\bfU\cdot\nabla\sigma) 
    -\bfU\cdot\nabla\partial^\alpha\sigma
    -\partial^\alpha \bfv\cdot\nabla\sigma, 
    \\
    &\mathcal{J}_2^\alpha:=\partial^\alpha
    (H\nabla\cdot\bfv) 
    -H\nabla\cdot\partial^\alpha\bfv 
    -\partial^\alpha \sigma\nabla\cdot\bfv, 
    \\
    &\mathcal{J}_3^\alpha:=\partial^\alpha
    ((\bfU\cdot\nabla)\bfv) 
    -(\bfU\cdot\nabla)\partial^\alpha\bfv 
    -(\partial^\alpha \bfv\cdot\nabla)\bfv, 
    \\
    &\mathcal{J}_4^\alpha:=\partial^\alpha
    (H\nabla\sigma) 
    -H\nabla\partial^\alpha\sigma 
    -\partial^\alpha \sigma\nabla\sigma, 
    \\
    &\mathcal{J}_5^\alpha:=\partial^{\alpha}
    (H^2\nabla \theta) 
    - H^2\nabla\partial^\alpha \theta
    -2H\partial^\alpha \sigma\nabla\theta, 
    \\
    &\mathcal{J}_6^\alpha:=\partial^\alpha
    (\bfU\cdot\nabla\theta) 
    -\partial^\alpha\bfv\cdot\nabla\theta 
    -\bfU\cdot\nabla\partial^\alpha 
    \theta, 
    \\
    &\mathcal{J}_7^\alpha:=\partial^\alpha
    (\bfv\cdot\nabla\bar h) 
    -\partial^\alpha\bfv\cdot\nabla\bar h, 
    \\
    &\mathcal{J}_8^\alpha:=\partial^\alpha
    ((\bfv\cdot\nabla)\bar\bfu) 
    -(\partial^\alpha\bfv\cdot\nabla)\bar\bfu, 
    \\
    &\mathcal{J}_9^\alpha:=\partial^\alpha
    (\sigma\nabla\cdot\bar\bfu) 
    -\partial^\alpha\sigma\nabla\cdot\bar\bfu, 
    \\
    &\mathcal{J}_{10}^\alpha:=\partial^\alpha
    (\sigma\nabla\bar h) 
    -\partial^\alpha\sigma\nabla\bar h, 
    \\
    &\mathcal{J}_{11}^\alpha:=\partial^\alpha
    (\sigma^2\nabla\bar S) 
    -2\sigma\nabla\bar S\partial^\alpha\sigma, 
    \\
    &\mathcal{J}_{12}^\alpha:=\partial^\alpha
    (\sigma\bar h\nabla\bar S) 
    -\bar h\nabla\bar S(\partial^\alpha\sigma), 
    \\
    &\mathcal{J}_{13}^\alpha:=\partial^\alpha
    (\bfv\cdot\nabla\bar S) 
    -(\partial^\alpha\bfv)\cdot\nabla\bar S,  
\end{align*} 
whose expansions and $L^2$ estimates are given in Appendix \ref{App:CommExpanded}. The essential point of this decomposition is that the remainder terms contain derivatives of at most order $|\alpha|=k$ of the perturbation variables. Using these, we define three terms, 

\begin{align*}
    &\mathcal{F}^\alpha:=-\mathcal{J}_1^\alpha
    -a\mathcal{J}_2^\alpha 
    -\mathcal{J}_7^\alpha
    -a\mathcal{J}_9^\alpha, 
    \\
    &\mathcal{G}^\alpha:=-\mathcal{J}_3^\alpha
    -\frac{1}{a}\mathcal{J}_4^\alpha 
    +\frac{1}{2a\gamma}\mathcal{J}_5^\alpha 
    -\mathcal{J}_8^\alpha
    -\frac{1}{a}\mathcal{J}_{10}^\alpha 
    +\frac{1}{2a\gamma}\mathcal{J}_{11}^\alpha 
    +\frac{1}{a\gamma}\mathcal{J}_{12}^\alpha, 
    \\
    &\mathcal{H}^\alpha:=-\mathcal{J}_6^\alpha 
    -\mathcal{J}_{13}^\alpha, 
\end{align*} 
along with remainder terms, 

\begin{align*}
    R^\alpha_1:=&-\partial^\alpha\bfv\cdot\nabla H 
    -a\partial^\alpha \sigma\nabla\cdot\bfU
    +\mathcal{F}^\alpha , 
    \\
    R^\alpha_2:=&-(\partial^\alpha \bfv\cdot\nabla)\bfU 
    -\frac{1}{a}\partial^\alpha \sigma\nabla H 
    +\frac{H}{a\gamma}\partial^\alpha\sigma 
    \nabla S+\mathcal{G}^\alpha , 
    \\
    R^\alpha_3:=&-\partial^\alpha\bfv\cdot\nabla S 
    +\mathcal{H}^\alpha, 
\end{align*} 
allowing us to rewrite the differentiated perturbation equations, 

\begin{align*}
    &(\partial^\alpha\sigma)_t
    +\bfU\cdot\nabla\partial^\alpha\sigma
    +aH\nabla\cdot\partial^\alpha\bfv=
    R^\alpha_1, 
    \\ 
    &(\partial^\alpha\bfv)_t
    +(\bfU\cdot\nabla)\partial^\alpha\bfv
    +\frac{1}{a}H\nabla\partial^\alpha\sigma 
    -\frac{1}{2a\gamma}H^2\nabla\partial^\alpha 
    \theta=R^\alpha_2 , 
    \\
    &(\partial^\alpha \theta)_t+\bfU\cdot\nabla 
    \partial^\alpha \theta = R^\alpha_3. 
\end{align*} 
Applying the same multipliers and simplifications as in Section \ref{Subsection:ZerothEnergy} yields the following differential relation on the $k$-th order energy. Details are shown in Appendix \ref{App:HigherOrderEnergy}.

\begin{align}
    &\frac{d}{dt} E_k(t) 
    = \frac{1}{2}\sum_{|\alpha|=k}\int_{\mathbb{R}^3} 
    e_\alpha \nabla\cdot\bfU \, dx
    +\sum_{|\alpha|=k}\biggl[\int_{\mathbb{R}^3}
    \frac{aH}{2\gamma}
    \partial^\alpha \theta 
    \partial^\alpha\sigma\nabla\cdot\bfU dx
    \nonumber\\
    &-a\int_{\mathbb{R}^3}\nabla\cdot\left(H\left( 
    \partial^\alpha\sigma-\frac{H}{2\gamma}
    \partial^\alpha\theta\right)\partial^\alpha\bfv 
    \right)dx
    \nonumber\\
    &+a\int_{\mathbb{R}^3}
    \left(\partial^\alpha\sigma 
    -\frac{H}{\gamma}\partial^\alpha\theta\right)
    \partial^\alpha\bfv\cdot\nabla H dx
    \nonumber\\
    &+\int_{\mathbb{R}^3}
    \left(\partial^\alpha\sigma-\frac{H}{2\gamma} 
    \partial^\alpha\theta\right)R_1^\alpha dx
    +a^2\int_{\mathbb{R}^3}
    \partial^\alpha\bfv\cdot R_2^\alpha dx
    \nonumber\\
    &+\int_{\mathbb{R}^3}\left(\lambda
    \partial^\alpha \theta 
    -\frac{H}{2\gamma}\partial^\alpha\sigma
    \right)R_3^\alpha dx\biggr]. 
    \label{Eqn:ddtEkUnsimplifiedText}
\end{align} 
Each term in (\ref{Eqn:ddtEkUnsimplifiedText}) is then bounded as shown in Appendix \ref{App:HigherOrderEnergy}. 

\begin{align}
    \frac{d}{dt}E_k(t)&\lesssim \biggl(\lVert\nabla\bfv 
    \rVert_{L^\infty}+\lVert\nabla\sigma 
    \rVert_{L^\infty}+\lVert\nabla\theta 
    \rVert_{L^\infty}+\lVert\nabla\bar\bfu 
    \rVert_{L^\infty}+\lVert\nabla\bar h 
    \rVert_{L^\infty}\nonumber\\
    +&\lVert\nabla \bar S 
    \rVert_{L^\infty}\biggr)E_k(t)
    +\sum_{|\alpha|=k}\left(
    \lVert\mathcal{F}^\alpha\rVert_{L^2} 
    +\lVert\mathcal{G}^\alpha\rVert_{L^2}
    +\lVert\mathcal{H}^\alpha\rVert_{L^2}\right) 
    E_k(t)^{1/2}.  
    \label{Eqn:ddtEk}
\end{align}
Finally, the evolution of the cumulative N-th order energy is determined by combining (\ref{Eqn:E0_Evolution}) with 
(\ref{Eqn:ddtEk}). Indeed, by Sobolev embedding, the decay of the first derivatives of $\bar\bfu$, $\bar h$ from Assumption \ref{Ass:A1}, and the isentropic nature of the background solution, letting $N\geq 4$ yields  

\begin{align*}
    \frac{d}{dt}\mathcal{E}_N(t)\lesssim &
    \left((1+t)^{-1} +\mathcal{E}_N(t)^{1/2}\right)\mathcal{E}_N(t)
    \nonumber\\
    +&\mathcal{E}_N(t)^{1/2}\sum_{1\leq|\alpha|\leq N}\left(
    \lVert\mathcal{F}^\alpha\rVert_{L^2} 
    +\lVert\mathcal{G}^\alpha\rVert_{L^2}
    +\lVert\mathcal{H}^\alpha\rVert_{L^2}\right) . 
\end{align*} 
It remains to bound the commutator norms. The estimates in Appendix \ref{App:CommExpanded} imply 

\begin{align}
    &\sum_{1\leq|\alpha|\leq N}(\lVert\mathcal{F}^\alpha 
    \rVert_{L^2}+\lVert\mathcal{G}^\alpha\rVert_{L^2}
    +\lVert\mathcal{H}^\alpha\rVert_{L^2}) 
    \lesssim (1+t)^{2B_{N+1}^*}
    (\mathcal{E}_N(t)^\frac{1}{2}
    +\mathcal{E}_N(t))
    \nonumber\\&+\mathcal{E}_N(t)
    +\mathcal{E}_N(t)^{3/2}, 
    \label{Eqn:FGH_L2Bounds}
\end{align} 
which gives the desired energy relation, 

\begin{align*}
    &\frac{d}{dt}\mathcal{E}_N(t) 
    \lesssim(1+t)^{-1}\mathcal{E}_N(t) 
    +\mathcal{E}_N(t)^{3/2} 
    \nonumber\\
    &+\mathcal{E}_N(t)^{1/2}\biggl( 
    (1+t)^{2B_{N+1}^*}
    (\mathcal{E}_N(t)^\frac{1}{2}
    +\mathcal{E}_N(t))
    +\mathcal{E}_N(t)
    +\mathcal{E}_N(t)^{3/2}
    \biggr). 
\end{align*} 
This may be simplified as 

\begin{align*}
    \frac{d}{dt}\mathcal{E}_N(t) 
    \lesssim 
    (1+t)^{2B_{N+1}^*} 
    \mathcal{E}_N(t) 
    +(1+t)^{2B_{N+1}^*} 
    \mathcal{E}_N(t)^{3/2} 
    +\mathcal{E}_N(t)^2. 
\end{align*} 

\section{Local Existence and Lifespan Estimate} 

Using the estimates from Section 3, we now prove Theorem \ref{Thm:PerturbedSoln}. This will be established in two steps. First, we establish local existence of (\ref{Eqn:Perturbed1})--(\ref{Eqn:Perturbed3}) even in the presence of the point vacuum at zero and at spatial infinity. Indeed, noting the uniform bounds derived in Sections 2 and 3, local existence follows as a direct application of \cite{Kato1975}. 

Once local existence has been established, we derive the stability timespan claimed in Theorem \ref{Thm:PerturbedSoln}. Using this, we close our bootstrap assumptions and continue the solution until $T^\varepsilon$. 

\subsection{Local Existence of the Perturbation} 

We first recall the perturbed system, 
\begin{align*}
    &\sigma_t
    +\bfU\cdot\nabla\sigma
    +aH\nabla\cdot\mathbf v
    =-\mathbf v\cdot\nabla\bar h
    -a\sigma\nabla\cdot\bar\bfu, 
    \\
    &\mathbf v_t
    +(\bfU\cdot\nabla)\mathbf v
    +\frac{1}{a}H\nabla\sigma 
    -\frac{H^2}{2a\gamma}\nabla\theta 
    =-(\bfv\cdot\nabla)\bar\bfu
    -\frac{1}{a}\sigma\nabla\bar h
    +\frac{H^2-\bar h^2}{2a\gamma}\nabla\bar S, 
    \\
    &\theta_t+\bfU\cdot\nabla \theta= 
    -\bfv\cdot\nabla\bar S. 
\end{align*} 
We seek to put this in symmetric-hyperbolic form. Define 
\begin{align*}
    &B_j:=\begin{bmatrix}
        U_j&aHe_j^\top&0\\
        \frac{1}{a}He_j&U_jI_3 
        &-\frac{1}{2a\gamma}H^2 e_j\\
        0&0&U_j
    \end{bmatrix}, 
    \\
    &F:=\begin{bmatrix}
        -\bfv\cdot\nabla\bar h-a\sigma\nabla 
        \cdot\bar\bfu \\
        -(\bfv\cdot\nabla)\bar\bfu 
        -\frac{1}{a}\sigma\nabla\bar h 
        +\frac{1}{2a\gamma}(H^2-\bar h^2) 
        \nabla\bar S \\
        -\bfv\cdot\nabla\bar S 
    \end{bmatrix}. 
\end{align*}
Recalling that 
\begin{align*}
    \mathcal{W}=\begin{bmatrix}
        \sigma\\\bfv\\\theta
    \end{bmatrix}\in\mathbb{R}^5, 
\end{align*} 
we obtain 
\begin{align}
    \partial_t\mathcal{W}+\sum_{j=1}^3 
    B_j(t,x,\mathcal{W})\partial_{x_j}
    \mathcal{W}=F(t,x,\mathcal{W}). 
    \label{Eqn:unsymmetrized}
\end{align} 
We then define its symmetrizer, 
\begin{align}
    A_0(H):=\begin{bmatrix}
        1&0&-\frac{1}{2\gamma}H\\
        0&a^2I_3&0\\
        -\frac{1}{2\gamma}H&0&\lambda 
    \end{bmatrix}.  
    \label{Eqn:symmetrizer}
\end{align} 
We then note from Assumptions \ref{Ass:A1}--\ref{Ass:A2} that $\lVert H_0\rVert_{L^\infty}<5C_0$ for $\varepsilon\ll1$. Furthermore, $A_0(H)$ is uniformly positive definite for $|H|\leq5C_0$. Applying (\ref{Eqn:symmetrizer}) to (\ref{Eqn:unsymmetrized}) yields 
\begin{align}
    &A_0(H)\partial_t\mathcal{W}+ 
    \sum_{j=1}^3A_j(t,x,\mathcal{W}) 
    \partial_{x_j}\mathcal{W} 
    =A_0(H)F(t,x,\mathcal{W}), 
    \label{Eqn:symmetrizedSystem}
\end{align} 
where 
\begin{align*}
    A_j(t,x,\mathcal{W})=\begin{bmatrix}
        U_j&aHe_j^\top&-\frac{1}{2\gamma}HU_j\\ 
        aHe_j&a^2U_jI_3&-\frac{a}{2\gamma}H^2e_j\\ 
        -\frac{1}{2\gamma}HU_j&-\frac{a}{2\gamma} 
        H^2e_j^\top&\lambda U_j 
    \end{bmatrix}, 
\end{align*} 
which are symmetric. Hence, (\ref{Eqn:symmetrizedSystem}) forms a quasilinear symmetric hyperbolic system. As previously established, $A_0(H)$ is uniformly positive definite. Furthermore, $A_0$, $A_j$, and $A_0F$ are all smooth functions of $\mathcal{W}$ and have the necessary regularity in $(t,x)$. Furthermore, $F$ vanishes whenever $\mathcal{W}=0$. Since $N\geq 4$, local existence for some positive time follows directly from \cite{Kato1975}. 

\subsection{Closure of the Bootstrap and Lifespan Estimate} 

In the previous subsection, we established local existence on some positive time interval, $[0,T_0]$. Let $T^\varepsilon_{\mathrm{life}}$ denote its maximal lifespan. It remains to show that $T^\varepsilon_{\mathrm{life}}\geq T^\varepsilon$ as defined in Theorem \ref{Thm:PerturbedSoln}. In this section, we close our estimates and determine a lower bound on $T^\varepsilon$. First, consider any interval $[0,T]$ where

\begin{align*}
    T < \min\{T^\varepsilon,T^\varepsilon
    _{\mathrm{life}}\}. 
\end{align*} 
Assume that the initial perturbation energy satisfies 

\begin{align*}
    \mathcal{E}_N(0)\leq C_2\varepsilon^2
\end{align*}
for some $C_2>0$. Then, set $\delta>0$ as in Theorem \ref{Thm:PerturbedSoln}. By coercivity and Sobolev embedding, there exists some $C_8\geq1$ such that whenever $\lVert H\rVert_{L^\infty}\leq5C_0$, the following bound holds, 
\begin{align}
    &\lVert\bfv\rVert_{W^{k,\infty}}
    +\lVert\sigma\rVert_{W^{k,\infty}}
    +\lVert\theta\rVert_{W^{k,\infty}}
    \leq C_8\mathcal{E}_N(t)^{1/2}.
    \label{Eqn:FinalSobolevBound}
\end{align} 
Using this, we define 
\begin{align*}
    \tilde \delta
    :=\min\left\{\frac{C_0}{2C_8},
    \frac{\delta}{2C_8}\right\}, 
\end{align*} 
and let 
\begin{align}
    \varepsilon<\frac{\tilde\delta}{\sqrt{2C_2}} 
    \label{Eqn:smallEpsilon}. 
\end{align}
We then impose the following bootstrap assumptions,   

\begin{align}
    \mathcal{E}_N(t)\leq\tilde\delta^2,
    \qquad
    \left\lVert H\right\rVert_{L^\infty([0,T]\times\mathbb{R}^3)}
    \leq5C_0.
    \label{Eqn:EnergyBoot}
\end{align}
This gives us 

\begin{align*}
    &\mathcal{E}_N(t)^\frac{3}{2}\leq 
    \tilde\delta \mathcal{E}_N(t), 
    \\
    &\mathcal{E}_N(t)^2\leq \tilde\delta ^2 
    \mathcal{E}_N(t), 
\end{align*} 
implying that 

\begin{align*}
    \frac{d}{dt}\mathcal{E}_N(t)&\leq 
    C_3(1+t)^{2B_{N+1}^*}\mathcal{E}_N(t) 
    +C_4\tilde\delta (1+t)^{2B_{N+1}^*}
    \mathcal{E}_N(t) 
    +C_5\tilde\delta ^2\mathcal{E}_N(t), 
\end{align*} 
for some $C_3,C_4,C_5>0$. Furthermore, since $B_{N+1}^*>0$, 

\begin{align*}
    \frac{d}{dt}\mathcal{E}_N(t)\leq 
    C_6(1+t)^{2B_{N+1}^*}\mathcal{E}_N(t), 
\end{align*} 
for some constant $C_6>0$. Applying Gr\"onwall's inequality yields   

\begin{align*}
    \mathcal{E}_N(t)&\leq C_2\varepsilon^2 
    \exp\left({C_6\int_{0}^t(1+s)
    ^{2B_{N+1}^*}ds}\right)
    \nonumber\\
    &= C_2\varepsilon^2\exp
    \left(\frac{C_6}
    {2B_{N+1}^*+1}\left(\left(1+t\right)
    ^{2B_{N+1}^*+1} 
    -1\right)\right). 
\end{align*} 
We then define $T^\varepsilon$ by the maximal time satisfying  

\begin{align*}
    C_2\varepsilon^2\exp\left(\frac{C_6}
    {2B_{N+1}^*+1}\left(\left(1+t\right)
    ^{2B_{N+1}^*+1} 
    -1\right)\right) \leq \frac{\tilde\delta ^2}{2}. 
\end{align*} 
From the monotonicity of the exponential and the smallness of $\varepsilon$ imposed by (\ref{Eqn:smallEpsilon}), $T^\varepsilon$ is determined by  

\begin{align}
    &\exp\left(\frac{C_6}
    {2B_{N+1}^*+1}\left(\left(1+T^\varepsilon
    \right)^{2B_{N+1}^*+1} 
    -1\right)\right) = \frac{\tilde\delta ^2}{2C_2\varepsilon^2}, 
    \nonumber\\
    &\frac{C_6}{2B_{N+1}^*+1}
    \left(\left(1+T^\varepsilon\right)^{2B_{N+1}^*+1} 
    -1\right)= \log\left(\frac{\tilde\delta ^2}
    {2C_2\varepsilon^2} \right),  
    \nonumber\\
    &T^\varepsilon=\left(\frac{2B_{N+1}^*+1}{C_6}
    \log\left(\frac{\tilde\delta ^2}
    {2C_2\varepsilon^2} 
    \right)+1\right)^\frac{1}{2B_{N+1}^*+1}-1.  
    \label{Eqn:T_epsilon_Explicit}
\end{align} 
Then, using coercivity and Sobolev embedding, we note that 
\begin{align*}
    \left\lVert\sigma(t)\right\rVert_{L^\infty}
    \leq C_8\mathcal E_N(t)^{1/2}
    \leq C_8\tilde\delta
    \leq\frac{C_0}{2}.
\end{align*}
Additionally, the background satisfies 
\begin{align*}
    \left\lVert \bar h(t) \right\rVert 
    _{L^\infty}\leq 4C_0. 
\end{align*}
Thus, $\left\lVert H\right\rVert_{L^\infty([0,T]
\times\mathbb{R}^3)}\leq\frac{9}{2}C_0$, improving on (\ref{Eqn:EnergyBoot}). This completes the proof of the lower bound on $T^\varepsilon$: 

\begin{align}
    T^\varepsilon \gtrsim \left(\log \frac{1}
    {\varepsilon}\right)^{\mathcal{C}_N}, 
    \label{Eqn:T_epsilon_final}
\end{align} 
where 

\begin{align*}
    \mathcal{C}_N := \frac{1}{2B^*_{N+1}+1}. 
\end{align*}
It remains to prove that the solution exists until $T^\varepsilon$. Using (\ref{Eqn:EnergyBoot}), we note that 
\begin{align*}
    \lVert\mathcal W(t)\rVert_{H^N}
    \leq C_7\mathcal E_N(t)^{1/2}
    \leq C_7\tilde\delta, 
\end{align*} 
for some $C_7>0$. Furthermore, define the quantities 
\begin{align}
    M&:=C_7\tilde\delta, 
    \label{Eqn:M}
    \\
    K_\varepsilon&:=\sup_{0\leq t\leq T^\varepsilon+1}\left\lVert 
    \mathcal{B}(t,\cdot) \right\rVert 
    _{C_b^{N+1}(\mathbb{R}^3)}. 
    \label{Eqn:K} 
\end{align} 
We then note that \cite{Kato1975} implies the local existence of a solution to (\ref{Eqn:Perturbed1})--(\ref{Eqn:Perturbed3}) whose lifespan is bounded below in terms of $M,K_\varepsilon$, and the uniform coercivity bound of the symmetrizer. Let $\tau_\varepsilon>0$ denote a uniform positive lower bound for the local lifespan supplied by \cite{Kato1975} corresponding to the bounds (\ref{Eqn:LambdaCoercivity}), (\ref{Eqn:M})--(\ref{Eqn:K}). Starting from $t=0$, we apply the local existence theorem supplied by \cite{Kato1975} to (\ref{Eqn:Perturbed1})--(\ref{Eqn:Perturbed3}) repeatedly over intervals of size $\tau_\varepsilon$. The preceding bootstrap estimates allow this process to be iterated for all $t<T^\varepsilon$. We deduce that $T^\varepsilon_{\mathrm{life}} \geq T^\varepsilon$. Therefore, 

\begin{align*}
    T^\varepsilon_{\mathrm{life}} \geq  
    T^\varepsilon \gtrsim \left(\log \frac{1}
    {\varepsilon}\right)^{\mathcal{C}_N}. 
\end{align*}  
Finally, since $N>k+3/2$, Sobolev
embedding and (\ref{Eqn:FinalSobolevBound}) imply 
\begin{align*}
    &\lVert \bfU-\bar\bfu\rVert_{W^{k,\infty}}
    +\lVert H-\bar h\rVert_{W^{k,\infty}}
    +\lVert\theta\rVert_{W^{k,\infty}}
    \leq C_8\mathcal{E}_N(t)^{1/2}
    \leq C_8\tilde\delta 
    \leq\frac{\delta}{2}<\delta.
\end{align*}
This completes the proof of Theorem \ref{Thm:PerturbedSoln}.

\begin{appendices}

\section{Algebraic Identities for Higher Radial Derivatives}
\label{App:HigherRadialAlgebra}

\subsection{Derivation of the characteristic identities}
\label{App:CharacteristicIdentities}

For $j\geq2$, we provide the algebraic details leading to the identities

\begin{align*}
    \partial_1(\partial_r^j z)=
    -I_j^z\partial_r^j z+\Lambda_j^z+Q_j^z+S_j^z
\end{align*}
and

\begin{align*}
    \partial_2(\partial_r^j w)=
    -I_j^w\partial_r^j w+\Lambda_j^w+Q_j^w+S_j^w.
\end{align*}
We begin by inserting (\ref{Eqn:Commutator_Relation}) into (\ref{Eqn:partial_1z})--(\ref{Eqn:partial_2w}).   

\begin{align*}
    &\partial_1\left(\partial_r^jz\right)= 
    \frac{\gamma-1}{4}\partial_r^j\left( 
    r^{-1}(w^2-z^2)\right)-\sum_{k=0}^{j-1} 
    \binom{j}{k}\partial_r^{j-k}c_1
    \partial_r^{k+1}z 
    \nonumber\\
    &= 
    -\frac{\gamma-1}{4}\partial_r^j(r^{-1}z^2) 
    +\frac{\gamma-1}{4}\partial_r^j\left( 
    r^{-1}w^2\right)-\sum_{k=0}^{j-1} 
    \binom{j}{k}\partial_r^{j-k}c_1
    \partial_r^{k+1}z 
    \nonumber\\
    &=-\frac{\gamma-1}{4}\sum_{k=0}^j 
    \binom{j}{k}\partial_r^{j-k}(r^{-1}) 
    \partial_r^k(z^2)+\frac{\gamma-1}{4} 
    \partial_r^j(r^{-1}w^2)-\sum_{k=0}^{j-1} 
    \binom{j}{k}\partial_r^{j-k}c_1 
    \partial_r^{k+1}z . 
\end{align*} 
We next examine each piece individually and extract the highest order derivative in $z$. We begin with the first term, 

\begin{align*}
    &\sum_{k=0}^j 
    \binom{j}{k}\partial_r^{j-k}(r^{-1}) 
    \partial_r^k (z^2)=
    r^{-1}\sum_{k=0}^j\binom{j}{k}\partial_r^kz
    \partial_r^{j-k}z+
    \sum_{k=0}^{j-1}\binom{j}{k} 
    \partial_r^{j-k}(r^{-1})\partial_r^k(z^2)
    \nonumber\\
    &=2r^{-1}z\partial^j_rz+r^{-1}\sum_{k=1}^{j-1} 
    \binom{j}{k}\partial_r^kz\partial_r^{j-k}z 
    +\sum_{k=0}^{j-1}\binom{j}{k}\partial_r^{j-k}
    (r^{-1})\partial_r^k (z^2) 
    \nonumber\\
    &=2r^{-1}z\partial_r^jz+jr^{-1}\partial_rz 
    \partial_r^{j-1}z+z^2\partial_r^j(r^{-1}) 
    +j\partial_r(r^{-1})\partial_r^{j-1}(z^2) 
    \nonumber\\
    &+\sum_{k=1}^{j-2}\binom{j}{k}\left(r^{-1} 
    \partial_r^kz\partial_r^{j-k}z 
    +\partial_r^k(z^2)\partial_r^{j-k}(r^{-1})
    \right). 
\end{align*} 
Then the third term, 

\begin{align*}
    &\sum_{k=0}^{j-1}\binom{j}{k}\partial_r^{j-k}
    c_1\partial_r^{k+1}z=\partial_rz
    \partial_r^j\left(\frac{3-\gamma}{4}w
    +\frac{\gamma+1}{4}z\right)+\sum_{k=1}^{j-1} 
    \binom{j}{k}\partial_r^{j-k}c_1
    \partial_r^{k+1}z 
    \nonumber\\
    &=\frac{\gamma+1}{4}\partial_rz\partial_r^jz 
    +\frac{3-\gamma}{4}\partial_rz\partial_r^jw 
    +\sum_{k=1}^{j-1}\binom{j}{k}\partial_r^{j-k}
    c_1\partial_r^{k+1}z 
    \nonumber\\
    &=\frac{\gamma+1}{4}\partial_rz\partial_r^jz
    +j\partial_rc_1\partial_r^jz
    +\frac{3-\gamma}{4}\partial_rz\partial_r^jw 
    +\sum_{k=1}^{j-2}\binom{j}{k}\partial_r^{j-k}
    c_1\partial_r^{k+1}z 
    \nonumber\\
    &=\partial_r^jz\left(\frac{\gamma+1}{4}(j+1)
    \partial_rz+\frac{3-\gamma}{4}
    j\partial_rw\right)+\frac{3-\gamma}{4} 
    \partial_rz\partial_r^jw 
    +\sum_{k=1}^{j-2}\binom{j}{k}
    \partial_r^{j-k}c_1 \partial_r^{k+1}z, 
\end{align*} 
which, after utilizing (\ref{Eqn:S_j^z})--(\ref{Eqn:I_j^z}), reduces to 

\begin{align*}
    \partial_1(\partial_r^jz)=-I^z_j\partial_r^jz 
     +\Lambda^z_j+Q^z_j+S^z_j . 
\end{align*} 
The identity for $w$ follows from the symmetry of (\ref{Eqn:partial_1z})--(\ref{Eqn:partial_2w}) under the substitutions 

\begin{align*}
    z&\leftrightarrow w, 
    \\
    c_1&\leftrightarrow c_2,  
    \\
    \partial_1&\leftrightarrow\partial_2,  
    \\
    \psi&\leftrightarrow \eta . 
\end{align*} 
Therefore, 

\begin{align*}
    \partial_2(\partial_r^jw)=-I^w_j\partial_r^jw 
     +\Lambda^w_j+Q^w_j+S^w_j , 
\end{align*} 
where the definitions of $I_j^w,\Lambda_j^w,Q_j^w,S_j^w$ in (\ref{Eqn:S_j^w})--(\ref{Eqn:I_j^w}) precisely follow those of (\ref{Eqn:S_j^z})--(\ref{Eqn:I_j^z}) under the above symmetries. This completes the derivation.

\subsection{Proof of the \(Q_j\) bound}
\label{App:QjBound}

Due to the symmetric nature of the isentropic system, it suffices to determine a bound for $Q_j^w$, which by symmetry will hold for $Q_j^z$. We proceed term by term, 

\begin{align}
    &\left\lvert\frac{jw_r\partial_r^{j-1}w}{r}
    \right\rvert\leq j\mathcal{M}_2(t)
    \mathcal{M}_{j-1}(t) , 
    \label{Eqn:Qj2}
    \\
    &|w^2\partial_r^j(r^{-1})|=\left|w^2\frac{j!}{r^{j+1}} 
    \right|\leq \frac{j}{2}\mathcal{M}_2(t) 
    \mathcal{M}_{j-1}(t) , 
    \label{Eqn:Qj1}
\end{align} 
and,  

\begin{align*}
    &\left|\frac{j\partial_r^{j-1}(w^2)}{r^2} \right| 
    \leq j\sum_{k=0}^{j-1}\binom{j-1}{k} 
    \frac{\left|\partial_r^kw\partial_r^{j-k-1}w
    \right|}{r^2}. 
\end{align*} 
We first bound the endpoints, 

\begin{align*}
    &\left|\frac{w\partial_r^{j-1}w}{r^2}\right| 
    \leq \frac{1}{2}\mathcal{M}_2(t)
    \mathcal{M}_{j-1}(t) , 
\end{align*} 
then the interior,  

\begin{align*}
    \left|\frac{\partial_r^kw\partial_r^{j-k-1}w}
    {r^2}\right|\leq\mathcal{M}_{k+1}(t) 
    \mathcal{M}_{j-k}(t) . 
\end{align*} 
This gives us  

\begin{align}
    \left|\frac{j\partial_r^{j-1}(w^2)}{r^2} \right| 
    \leq j\mathcal{M}_2(t)\mathcal{M}_{j-1}(t) 
    +j\sum_{k=1}^{j-2}\binom{j-1}{k} 
    \mathcal{M}_{k+1}(t)\mathcal{M}_{j-k}(t) . 
    \label{Eqn:Qj3}
\end{align} 
We then bound the first summation. 

\begin{align*}
    &\left|r^{-1}\sum_{k=1}^{j-1}\binom{j}{k} 
    \partial_r^{j-k}z\partial_r^kz\right| 
    =\left|\frac{2jz_r\partial_r^{j-1}z}{r} 
    +\frac{1}{r}\sum_{k=2}^{j-2}\binom{j}{k}
    \partial_r^{j-k}z\partial_r^k z\right| 
    \\
    &\leq 2j\mathcal{M}_2(t)\mathcal{M}_{j-1}(t) 
    +\sum_{k=2}^{j-2}\binom{j}{k}\mathcal{M}_{j-k}(t) 
    \mathcal{M}_{k+1}(t) . 
\end{align*} 
It remains to estimate the second sum appearing in $Q_j^w$. We define 

\begin{align*}
    T_j^w:&=\sum_{k=0}^{j-1}\binom{j}{k} 
    \partial_r^{j-k}(r^{-1})\partial_r^k(z^2) 
    \\
    &=\sum_{k=0}^{j-1}\binom{j}{k}(-1)^{j-k} 
    (j-k)!r^{-(j-k+1)}\partial_r^k (z^2) , 
\end{align*} 
which is bounded by 

\begin{align*}
    |T_j^w|\leq \sum_{k=0}^{j-1}\binom{j}{k} 
    (j-k)!\sum_{l=0}^k\binom{k}{l}\left| 
    \frac{\partial_r^lz\partial_r^{k-l}z}{r^{j-k+1}}
    \right|. 
\end{align*} 
As before, we split into end cases. 

\paragraph{Case 1: $l=0$} 

We have  

\begin{align*}
    \left|\frac{z\partial_r^kz}{r^{j-k+1}}\right| 
    &=\left|\frac{z}{r^2}\right|\left|
    \frac{\partial_r^kz}{r^{j-k-1}}\right| 
    \leq \frac{1}{2(j-k-1)!}\mathcal{M}_2(t) 
    \mathcal{M}_{j-1}(t) , 
\end{align*} 
giving us  

\begin{align*}
    \binom{j}{k}(j-k)!\left|\frac{z\partial_r^kz}
    {r^{j-k+1}}\right|\leq\binom{j}{k}\frac{j-k}{2} 
    \mathcal{M}_2(t)\mathcal{M}_{j-1}(t) . 
\end{align*}

\paragraph{Case 2: $l=1$} 

We first observe that, for $k\geq 1$,  

\begin{align*}
    \left|\frac{z_r\partial_r^{k-1}z}{r^{j-k+1}} 
    \right|=\left|\frac{z_r}{r}\right|\left| 
    \frac{\partial_r^{k-1}z}{r^{j-k}}\right| 
    \leq \frac{1}{(j-k)!}\mathcal{M}_2(t) 
    \mathcal{M}_{j-1}(t) . 
\end{align*} 
This gives us  

\begin{align*}
    \binom{j}{k}(j-k)!\binom{k}{1} 
    \left|\frac{z_r\partial_r^{k-1}z}{r^{j-k+1}} 
    \right|\leq k\binom{j}{k}\mathcal{M}_2(t) 
    \mathcal{M}_{j-1}(t).  
\end{align*} 

\paragraph{Case 3: $2\leq l\leq k$} 

We then examine  

\begin{align*}
    \left|\frac{\partial_r^lz\partial_r^{k-l}z}
    {r^{j-k+1}}\right|&=|\partial_r^lz|\left|
    \frac{\partial_r^{k-l}z}{r^{j-k+1}}\right| 
    \\
    &\leq\frac{1}{(j-k+1)!}\mathcal{M}_l(t) 
    \mathcal{M}_{j+1-l}(t),  
\end{align*} 
which gives us   

\begin{align*}
    \binom{j}{k}(j-k)!\binom{k}{l}\left| 
    \frac{\partial_r^lz\partial_r^{k-l}z}
    {r^{j-k+1}}\right|\leq\binom{j}{k}
    \binom{k}{l}\frac{1}{(j-k+1)}\mathcal{M}_l(t) 
    \mathcal{M}_{j+1-l}(t). 
\end{align*} 
We now put all of these estimates together in order to bound $|T_j|$. We first note the combinatorial relations, 

\begin{align*}
    &\frac{1}{2}\sum_{k=0}^{j-1}\binom{j}{k}(j-k) 
    =j\cdot 2^{j-2}, 
    \\
    &\sum_{k=1}^{j-1}k\binom{j}{k}=j(2^{j-1}-1). 
\end{align*} 
This gives us  

\begin{align}
    |T_j^w|\leq j(3\cdot2^{j-2}-1)\mathcal{M}_2(t) 
    \mathcal{M}_{j-1}(t)+\sum_{k=2}^{j-1} 
    \sum_{l=2}^k\frac{\binom{j}{k}
    \binom{k}{l}}{j-k+1}\mathcal{M}_l(t) 
    \mathcal{M}_{j+1-l}(t) . 
    \label{Eqn:T_j_Bounds}
\end{align} 
Combining (\ref{Eqn:Qj2})--(\ref{Eqn:T_j_Bounds}) allows us to bound $Q_j^w(t)$---which, by symmetry, also holds for $Q_j^z(t)$. 

\begin{align*}
    Q_j(t)&\leq j\mathcal{M}_2(t)\mathcal{M}_{j-1}(t) 
    +\frac{j}{2}\mathcal{M}_2(t) 
    \mathcal{M}_{j-1}(t) 
    \nonumber\\
    &+j\mathcal{M}_2(t)\mathcal{M}_{j-1}(t) 
    +j\sum_{k=1}^{j-2}\binom{j-1}{k} 
    \mathcal{M}_{k+1}(t)\mathcal{M}_{j-k}(t) 
    \nonumber\\
    &+2j\mathcal{M}_2(t)\mathcal{M}_{j-1}(t) 
    +\sum_{k=2}^{j-2}\binom{j}{k}
    \mathcal{M}_{j-k}(t)\mathcal{M}_{k+1}(t) 
    \nonumber\\
    &+j(3\cdot2^{j-2}-1)\mathcal{M}_2(t)
    \mathcal{M}_{j-1}(t)
    +\sum_{k=2}^{j-1}\sum_{l=2}^k
    \frac{\binom{j}{k}\binom{k}{l}}{j-k+1}
    \mathcal{M}_l(t)\mathcal{M}_{j+1-l}(t). 
\end{align*} 
Collecting terms gives 

\begin{align*}
    Q_j(t)&\leq j\left(3\cdot2^{j-2}+4\right) 
    \mathcal{M}_2(t)\mathcal{M}_{j-1}(t) 
    +j\sum_{k=1}^{j-2}\binom{j-1}{k} 
    \mathcal{M}_{k+1}(t)\mathcal{M}_{j-k}(t) 
    \nonumber\\
    &+\sum_{k=2}^{j-2}\binom{j}{k}
    \mathcal{M}_{j-k}(t)\mathcal{M}_{k+1}(t) 
    +\sum_{k=2}^{j-1}\sum_{l=2}^k
    \frac{\binom{j}{k}\binom{k}{l}}{j-k+1}
    \mathcal{M}_l(t)\mathcal{M}_{j+1-l}(t), 
\end{align*} 
which can be written more cleanly as 

\begin{align*}
    Q_j(t)&\leq j\left(3\cdot2^{j-2}+3+j\right) 
    \mathcal{M}_2(t)\mathcal{M}_{j-1}(t) 
    \nonumber\\
    &+\sum_{k=2}^{j-2}(j-k+1)\binom{j}{k} 
    \mathcal{M}_{k+1}(t)\mathcal{M}_{j-k}(t) 
    \nonumber\\
    &+\sum_{k=2}^{j-1}\sum_{l=2}^k
    \frac{\binom{j}{k}\binom{k}{l}}{j-k+1}
    \mathcal{M}_l(t)\mathcal{M}_{j+1-l}(t). 
\end{align*} 
Utilizing our definition of $\mathcal{L}_j$ gives us the following bound, 

\begin{align*}
    &Q_j(t)\leq D_j^1\mathcal{L}_{j-1}(t)^2, 
\end{align*} 
for $j\geq 3$, where $D_j^1$ is some constant depending only on $j$. 

\subsection{Proof of the $S_j$ Bound}
\label{App:SjBound}

We proceed term by term. For the first term in (\ref{Eqn:S_j^w}), we have 

\begin{align}
    \left|\sum_{k=1}^{j-2}\binom{j}{k} 
    r^{-1}\partial_r^kw\partial_r^{j-k}w\right|
    \leq \sum_{k=1}^{j-2}\binom{j}{k} 
    \mathcal{M}_{k+1}(t)\mathcal{M}_{j-k}(t) . 
    \label{Eqn:Sj1}
\end{align} 
Mimicking the derivation for Equation (\ref{Eqn:T_j_Bounds}) bounds the second term,  

\begin{align}
    &\left|\sum_{k=1}^{j-2}\binom{j}{k} 
    \partial_r^k(w^2)\partial_r^{j-k}(r^{-1})\right| 
    \leq \sum_{k=1}^{j-2}\binom{j}{k}(j-k)! 
    \sum_{l=0}^k\binom{k}{l}\left| 
    \frac{\partial_r^lw\partial_r^{k-l}w}
    {r^{j-k+1}}\right| \nonumber
    \\
    &\leq j(3\cdot2^{j-2}-j-1)\mathcal{M}_2(t) 
    \mathcal{M}_{j-1}(t)+\sum_{k=2}^{j-2} 
    \sum_{l=2}^k \frac{\binom{j}{k}\binom{k}{l}}
    {j-k+1}\mathcal{M}_l(t)\mathcal{M}_{j+1-l}(t).  
    \label{Eqn:Sj2}
\end{align} 
We then examine the third term and note that $c_2$ is a linear combination of $w,z$ with positive weights summing to $1$. Therefore,  

\begin{align}
    \left|\sum_{k=1}^{j-2}\binom{j}{k} 
    \partial_r^{k+1}w\partial_r^{j-k}c_2\right| 
    \leq\sum_{k=1}^{j-2}\binom{j}{k} 
    \mathcal{M}_{k+1}(t)\mathcal{M}_{j-k}(t).  
    \label{Eqn:Sj3}
\end{align} 
Summing (\ref{Eqn:Sj1})--(\ref{Eqn:Sj3}) and noting the symmetry of the isentropic system gives us 

\begin{align}
    &S_j(t)\leq j(3\cdot2^{j-2}-j-1)\mathcal{M}_2(t) 
    \mathcal{M}_{j-1}(t)+\sum_{k=2}^{j-2} 
    \sum_{l=2}^k \frac{\binom{j}{k}\binom{k}{l}}
    {j-k+1}\mathcal{M}_l(t)\mathcal{M}_{j+1-l}(t) 
    \nonumber\\
    &+\sum_{k=1}^{j-2}\binom{j}{k}
    \mathcal{M}_{k+1}(t)\mathcal{M}_{j-k}(t) 
    +\sum_{k=1}^{j-2}\binom{j}{k}
    \mathcal{M}_{k+1}(t)\mathcal{M}_{j-k}(t),  
    \label{Eqn:S_j_Bound}
\end{align} 
which can then be simplified as 

\begin{align*}
    S_j(t)\leq D_j^2 \mathcal{L}_{j-1}(t)^2, 
\end{align*} 
for $j\geq 3$, where $D_j^2$ is some constant depending only on $j$. 

\section{Radial-to-Cartesian Derivative Estimates}
\label{App:RadialCartesian}

\subsection{Scalar radial functions}
\label{App:ScalarRadial}

\begin{lemma}\label{Lemma:CartRadScalar}
    Let $|\alpha|=k\geq1$, and let $f(x)=F(r)\in\mathbb{R}$ where $r=|x|$. Then for all $x\neq0$, 

\begin{align}
    \partial^\alpha f(x) = \sum_{j=1}^k 
    \frac{F^{(j)}(r)}{r^{k-j}}A_{\alpha,j}\left(
    \frac{x}{r}\right), 
    \label{Eqn:CartRadScalar}
\end{align} 
where the $A_{\alpha,j}$ are smooth and bounded with respect to their argument. 
\end{lemma}

\begin{proof}
    We first write our radial unit vector component-wise,  

\begin{align*}
    \omega:=\frac{x}{r}=\left[\frac{x_1}{r},
    \frac{x_2}{r},\frac{x_3}{r}\right], 
\end{align*} 
and define the following notation, 

\begin{align*}
    \partial_l:=\frac{\partial}{\partial x_l}. 
\end{align*}
This gives us  

\begin{align*}
    \partial_l r= \omega_l,  
\end{align*} 
and its partial derivative componentwise,  

\begin{align*}
    &\partial_l\omega_i=\frac{\delta_{il}
    -\omega_i\omega_l}{r}. 
\end{align*}
We proceed by induction. The base case is trivial. Indeed, 

\begin{align*}
    \partial_l f(x) = F'(r)\frac{x_l}{r}. 
\end{align*} 
Then, suppose by way of induction that (\ref{Eqn:CartRadScalar}) holds for all $|\alpha|\leq k$.  

\begin{align*}
    \partial^\alpha f(x) = \sum_{j=1}^k 
    \frac{F^{(j)}(r)}{r^{k-j}}A_{\alpha,j}\left(
    \frac{x}{r}\right). 
\end{align*} 
We then differentiate once more, 

\begin{align*}
    \partial_l\partial^\alpha f(x) &= 
    \sum_{j=1}^k
    \biggl((j-k)\frac{F^{(j)}(r)}{r^{(k+1)-j}} 
    \frac{x_l}{r}
    A_{\alpha,j}\left(\frac{x}{r}\right) 
    +\frac{F^{(j+1)}(r)}{r^{(k+1)-(j+1)}}\frac{x_l}{r}A_{\alpha,j} 
    \left(\frac{x}{r} \right) 
    \nonumber\\
    &+\frac{F^{(j)}(r)}{r^{k+1-j}}
    B_{\alpha,j,l}(\omega) \biggr), 
\end{align*}  
where 

\begin{align*}
    B_{\alpha,j,l}(\omega):=\sum_{i=1}^3 
    \frac{\partial A_{\alpha,j}}{\partial\omega_i} 
    (\omega)(\delta_{il}-\omega_i\omega_l) . 
\end{align*} 
Since $|x_l/r|\leq 1$, they can be absorbed into the $A_{\alpha,j}$---as can $B_{\alpha,j,l}$. This gives the desired form and completes the proof.
\end{proof}

\subsection{Vector-valued radial functions}
\label{App:VectorRadial}

\begin{lemma}\label{Lemma:CartRadVector}
    Let $|\alpha|=k\geq0$, and let 
    $G(x)=g(r)\omega\in\mathbb{R}^3$. Then for all $x\neq0$, 

\begin{align}
    \partial^\alpha G_i(x) 
    =\sum_{j=0}^k \frac{\partial_r^jg(r)}
    {r^{k-j}}A_{\alpha,i,j}(\omega), 
    \label{Eqn:CartRadVector}
\end{align} 
where the $A_{\alpha,i,j}$ are smooth and bounded with respect to their argument. 
\end{lemma} 

\begin{proof}
    Here, the base case starts at $k=0$. Indeed, we have trivially, 

    \begin{align*}
        G_i(x)=g(r)\omega_i. 
    \end{align*} 
    We then proceed to the inductive case. Suppose that the following holds for $|\alpha| = k$, 
    
    \begin{align*}
        \partial^\alpha G_i(x) 
        =\sum_{j=0}^k \frac{\partial_r^jg(r)}
        {r^{k-j}}A_{\alpha,i,j}(\omega). 
    \end{align*} 
    Differentiating once more gives 

    \begin{align*}
    &\partial_l\partial^\alpha G_i(x) 
    =\sum_{j=0}^k \partial_l\left(\frac{\partial_r^jg(r)}
    {r^{k-j}}A_{\alpha,i,j}(\omega)\right) 
    \\
    &=\sum_{j=0}^k
    \biggl((j-k)\frac{\partial_r^j g(r)}
    {r^{(k+1)-j}}\frac{x_l}{r}
    A_{\alpha,i,j}\left(\frac{x}{r}\right) 
    +\frac{\partial_r^{j+1} g(r)}{r^{(k+1)-(j+1)}}\frac{x_l}{r}A_{\alpha,i,j} 
    \left(\frac{x}{r} \right) 
    \nonumber\\
    &+\frac{\partial_r^j g(r)}{r^{k+1-j}}
    B_{\alpha,i,j,l}(\omega) \biggr). 
    \end{align*} 
    As in the proof of Lemma \ref{Lemma:CartRadScalar}, terms can be appropriately collected and re-indexed to be of the form in (\ref{Eqn:CartRadVector}). This completes the proof. 
    
\end{proof} 

We note that while the formulas in Appendix \ref{App:RadialCartesian} are stated for $x\neq0$, the corresponding $L^\infty$ estimates extend to the origin by smoothness.

\section{Commutator Relations}
\label{App:Comm} 

\subsection{Norms of Commutators}
\label{App:CommExpanded} 

The precise expansions of the $\mathcal{J}_i^\alpha$ are given below. 

\begin{align*}
    &\mathcal{J}_1^\alpha=\partial^\alpha \bar\bfu\cdot\nabla\sigma
    +\sum_{0<\beta<\alpha}\binom{\alpha}{\beta} 
    \partial^\beta \bfU\cdot\nabla
    \partial^{\alpha-\beta}\sigma, 
    \\
    &\mathcal{J}_2^\alpha=\partial^\alpha \bar h\nabla\cdot\bfv
    +\sum_{0<\beta<\alpha}\binom{\alpha}{\beta} 
    \partial^\beta H\nabla\cdot\partial^{\alpha-\beta} 
    \bfv, 
    \\
    &\mathcal{J}_3^\alpha=(\partial^\alpha\bar\bfu\cdot\nabla)\bfv
    +\sum_{0<\beta<\alpha}\binom{\alpha}{\beta} 
    (\partial^\beta \bfU\cdot\nabla)
    \partial^{\alpha-\beta}
    \bfv, 
    \\
    &\mathcal{J}_4^\alpha=\partial^\alpha\bar h\nabla\sigma+\sum_{0<\beta<\alpha}
    \binom{\alpha}{\beta} 
    \partial^\beta H\nabla\partial^{\alpha-\beta}
    \sigma, 
    \\
    &\mathcal{J}_5^\alpha=2H\partial^\alpha\bar h\nabla\theta
    +\sum_{0<\beta<\alpha}
    \binom{\alpha}{\beta}\left(  
    \partial^\beta(H^2)\nabla
    \partial^{\alpha-\beta}\theta 
    +(\partial^\beta H\partial^{\alpha-\beta} H)
    \nabla\theta\right) , 
    \\
    &\mathcal{J}_6^\alpha=\partial^\alpha\bar\bfu
    \cdot\nabla\theta+\sum_{0<\beta<\alpha}
    \binom{\alpha}{\beta}\partial^{\beta}\bfU
    \cdot\nabla\partial^{\alpha-\beta}\theta , 
    \\
    &\mathcal{J}_7^\alpha=\sum_{0<\beta\leq\alpha}
    \binom{\alpha}{\beta} 
    \partial^{\alpha-\beta}\bfv\cdot\nabla 
    \partial^\beta\bar h , 
    \\
    &\mathcal{J}_8^\alpha=\sum_{0<\beta\leq\alpha}
    \binom{\alpha}{\beta} 
    (\partial^{\alpha-\beta}\bfv\cdot\nabla) 
    \partial^\beta\bar \bfu , 
    \\
    &\mathcal{J}_9^\alpha=\sum_{0<\beta\leq\alpha}
    \binom{\alpha}{\beta} 
    \partial^{\alpha-\beta}\sigma 
    \nabla\cdot\partial^\beta\bar\bfu ,
    \\
    &\mathcal{J}_{10}^\alpha=\sum_{0<\beta\leq\alpha}
    \binom{\alpha}{\beta} 
    \partial^{\alpha-\beta}\sigma\nabla\partial^\beta 
    \bar h , 
    \\
    &\mathcal{J}_{11}^\alpha=\sigma^2\nabla
    \partial^\alpha\bar S 
    +\sum_{0<\beta<\alpha}\binom{\alpha}{\beta} 
    \left(\partial^\beta(\sigma^2) 
    \nabla\partial^{\alpha-\beta}\bar S 
    +\partial^\beta\sigma
    \partial^{\alpha-\beta}\sigma\nabla\bar 
    S\right) , 
    \\
    &\mathcal{J}_{12}^\alpha=\sum_{0\leq\beta<\alpha}
    \binom{\alpha}{\beta} 
    \partial^\beta\sigma\partial^{\alpha-\beta} 
    (\bar h\nabla\bar S) , 
    \\
    &\mathcal{J}_{13}^\alpha=\sum_{0\leq\beta<\alpha}
    \binom{\alpha}{\beta} 
    (\partial^\beta\bfv)\cdot\nabla
    \partial^{\alpha-\beta}\bar S. 
\end{align*} 
We then provide estimates of their $L^2$-norms. 

\subsubsection{Bounding $\mathcal{J}^\alpha_1$} 
\label{App:J1Bds}
We note that 

\begin{align*}
    \lVert\mathcal{J}_1^\alpha\rVert_{L^2} 
    &\leq\lVert\partial^\alpha\bar\bfu\rVert_{L^\infty} 
    \lVert\nabla\sigma\rVert_{L^2} 
    +\sum_{0<\beta<\alpha}\binom{\alpha}{\beta} 
    \lVert\partial^\beta\bar\bfu\rVert
    _{L^\infty}\lVert
    \nabla\partial^{\alpha-\beta} \sigma\rVert_{L^2} 
    \\
    &+\sum_{0<\beta<\alpha}\binom{\alpha}{\beta} 
    \lVert\partial^\beta\bfv\cdot
    \nabla\partial^{\alpha-\beta}\sigma\rVert_{L^2}. 
\end{align*}
We examine each term individually. For our first two terms, we have  

\begin{align*}
    &\lVert\partial^\alpha\bar\bfu\rVert_{L^\infty} 
    \lVert\nabla\sigma\rVert_{L^2} 
    \leq A_k^*(1+t)^{B_k^*}\mathcal{E}_N(t)^\frac{1}{2}, 
    \\
    &\lVert\partial^\beta\bar\bfu\rVert
    _{L^\infty}\lVert
    \nabla\partial^{\alpha-\beta} \sigma\rVert_{L^2} 
    \leq A_k^*(1+t)^{B_k^*}
    \mathcal{E}_N(t)^\frac{1}{2}.  
\end{align*} 
For the third term, we let $N\geq4$ and we split this into two cases. 

\paragraph{Case 1: }
$|\beta|\leq N-2$
\\
We note that $H^2(\mathbb{R}^3)
\hookrightarrow L^\infty(\mathbb{R}^3)$; therefore, 

\begin{align*}
    \lVert\partial^\beta\bfv\rVert_{L^\infty} 
    \lesssim \lVert \bfv \rVert_{H^N}. 
\end{align*} 
Furthermore, we note that 

\begin{align*}
    |\alpha|-|\beta|+1\leq |\alpha|\leq N. 
\end{align*} 
It follows that  

\begin{align*}
    \lVert \nabla\partial^{\alpha-\beta} 
    \sigma\rVert_{L^2}\leq \lVert\sigma
    \rVert_{H^N}, 
\end{align*} 
giving us 

\begin{align*}
    \lVert\partial^\beta\bfv\cdot
    \nabla\partial^{\alpha-\beta}\sigma\rVert_{L^2} 
    \lesssim \lVert\bfv\rVert_{H^N} 
    \lVert\sigma
    \rVert_{H^N}. 
\end{align*} 
\paragraph{Case 2: } $|\beta|\geq N-1$ 
\\
Since $0<\beta<\alpha$, this implies that $|\beta|=N-1$. Therefore, since $N\geq 4$, 

\begin{align*}
    \lVert \nabla\partial^{\alpha-\beta} 
    \sigma\rVert_{L^\infty}\lesssim 
    \lVert\sigma\rVert_{H^N} . 
\end{align*} 
Furthermore, 

\begin{align*}
    \lVert\partial^\beta\bfv\rVert_{L^2} 
    \leq \lVert \bfv \rVert_{H^N}. 
\end{align*} 
This gives us the same estimate, 

\begin{align*}
    \lVert\partial^\beta\bfv\cdot
    \nabla\partial^{\alpha-\beta}\sigma\rVert_{L^2} 
    \lesssim \lVert\bfv\rVert_{H^N} 
    \lVert\sigma
    \rVert_{H^N}, 
\end{align*} 
implying that  

\begin{align*}
    \lVert\partial^\beta\bfv\cdot
    \nabla\partial^{\alpha-\beta}\sigma\rVert_{L^2} 
    \lesssim \mathcal{E}_N(t). 
\end{align*} 
Therefore, we have the following bound, 

\begin{align*}
    \lVert\mathcal{J}_1^\alpha\rVert_{L^2} 
    \lesssim (1+t)^{B_k^*}\mathcal{E}_N(t)^\frac{1}{2} 
    +\mathcal{E}_N(t) . 
\end{align*} 

\subsubsection{Bounds for $\mathcal{J}^\alpha_2-\mathcal{J}^\alpha_4$} 
From the same process as in Appendix \ref{App:J1Bds}, we obtain 
\begin{align*}
    &\lVert\mathcal{J}^\alpha_2\rVert_{L^2} 
    \lesssim (1+t)^{B_k^*}\mathcal{E}_N(t)^\frac{1}{2} 
    +\mathcal{E}_N(t), 
    \\
    &\lVert\mathcal{J}^\alpha_3\rVert_{L^2} 
    \lesssim (1+t)^{B_k^*}\mathcal{E}_N(t)^\frac{1}{2} 
    +\mathcal{E}_N(t) , 
    \\
    &\lVert\mathcal{J}^\alpha_4\rVert_{L^2} 
    \lesssim (1+t)^{B_k^*}\mathcal{E}_N(t)^\frac{1}{2} 
    +\mathcal{E}_N(t) . 
\end{align*} 

\subsubsection{Bounding $\mathcal{J}^\alpha_5$} 
For $\mathcal{J}_5^\alpha$, we note that 

\begin{align*}
    \mathcal{J}_5^\alpha=2H\partial^\alpha
    \bar h\nabla\theta 
    +\sum_{0<\beta<\alpha}
    \binom{\alpha}{\beta}\left(  
    \partial^\beta(H^2)\nabla
    \partial^{\alpha-\beta}\theta 
    +(\partial^\beta H\partial^{\alpha-\beta} H)
    \nabla \theta \right) . 
\end{align*} 
We examine this term by term. 

\begin{align*}
    &\lVert H\partial^\alpha\bar h\nabla\theta
    \rVert_{L^2} 
    \lesssim \lVert\partial^\alpha\bar h\rVert_{L^\infty} 
    \lVert\nabla \theta\rVert_{L^2} 
    \lesssim (1+t)^{B_k^*}\mathcal{E}_N(t)^{1/2} . 
\end{align*} 
Then the second term, 

\begin{align*}
    &\sum_{0<\beta<\alpha}
    \binom{\alpha}{\beta} 
    (\partial^\beta(H^2)\nabla
    \partial^{\alpha-\beta}\theta)
    = \sum_{0<\beta<\alpha}
    \binom{\alpha}{\beta} 
    (\partial^\beta(\bar h^2+2\bar h\sigma
    +\sigma^2)\nabla\partial^{\alpha-\beta}\theta). 
\end{align*} 
We note that 

\begin{align*}
    \lVert\partial^\beta (\bar h^2)\rVert_{L^\infty}
    \lesssim\sum_{0\leq \kappa\leq 
    \beta}\lVert\partial^\kappa 
    \bar h\rVert_{L^\infty} 
    \lVert \partial^{\beta-\kappa}\bar 
    h\rVert_{L^\infty} 
    \lesssim (1+t)^{2B_k^*}. 
\end{align*} 
Therefore, 

\begin{align*}
    \left\lVert\sum_{0<\beta<\alpha}
    \partial^\beta(\bar h^2) 
    \nabla\partial ^{\alpha-\beta}\theta
    \right\rVert_{L^2} 
    \lesssim (1+t)^{2B_k^*}\mathcal{E}_N(t)
    ^\frac{1}{2}. 
\end{align*}
Then the next term, 

\begin{align*}
    \left\lVert\sum_{0<\beta<\alpha}
    \partial^\beta(\bar h 
    \sigma)\nabla\partial^{\alpha-\beta}\theta 
    \right\rVert_{L^2}\lesssim 
    (1+t)^{B_k^*}\mathcal{E}_N(t) . 
\end{align*} 
And the last term, 

\begin{align*}
    &\sum_{0<\beta<\alpha}
    \binom{\alpha}{\beta}
    \partial^\beta(\sigma^2) 
    \nabla\partial^{\alpha-\beta}\theta 
    =\sum_{0<\beta<\alpha}
    \binom{\alpha}{\beta}\nabla
    \partial^{\alpha-\beta}\theta  
    \sum_{0\leq\kappa\leq\beta} 
    \binom{\beta}{\kappa}
    \partial^\kappa \sigma 
    \partial^{\beta-\kappa}\sigma . 
\end{align*} 
In order to bound this, we proceed case by case. 

\paragraph{Case 1:} $|\kappa|\leq N-2, |\beta-\kappa|\leq N-2, |\alpha-\beta|+1\leq N-2$.  

In this case, since $N\geq 4$, we have the following by Sobolev embedding. 
\begin{align*}
    \lVert \partial^\kappa\sigma 
    \partial^{\beta-\kappa}\sigma 
    \nabla\partial^{\alpha-\beta}\theta\rVert_{L^2} 
    &\leq\lVert\partial^\kappa\sigma\rVert_{L^\infty} 
    \lVert\partial^{\beta-\kappa}\sigma\rVert_{L^\infty} 
    \lVert\nabla\partial^{\alpha-\beta}\theta
    \rVert_{L^2} 
    \nonumber\\
    &\lesssim \lVert\sigma\rVert_{H^N}^2 
    \lVert \theta \rVert_{H^N} . 
\end{align*} 

\paragraph{Case 2:} $|\kappa|\geq N-1$ 

We observe that 
\begin{align*}
    |\kappa|+|\beta-\kappa|+|\alpha-\beta|+1\leq N+1. 
\end{align*} 
Therefore, 

\begin{align*}
    |\beta-\kappa|+|\alpha-\beta|+1\leq 2, 
\end{align*} 
which implies 

\begin{align*}
    &|\beta-\kappa|\leq 2 , 
    \\
    &|\alpha-\beta|+1\leq 2. 
\end{align*} 
Thus, we put the infinity norm on these two terms, giving us 
\begin{align*}
    \lVert \partial^\kappa\sigma 
    \partial^{\beta-\kappa}\sigma 
    \nabla\partial^{\alpha-\beta}\theta\rVert_{L^2} 
    &\leq\lVert\partial^\kappa\sigma\rVert_{L^2} 
    \lVert\partial^{\beta-\kappa}\sigma\rVert_{L^\infty} 
    \lVert\nabla\partial^{\alpha-\beta}\theta
    \rVert_{L^\infty} 
    \nonumber\\
    &\lesssim \lVert\sigma\rVert_{H^N}^2 
    \lVert \theta \rVert_{H^N} . 
\end{align*} 

\paragraph{Case 3:} $|\beta-\kappa|\geq N-1$ 

In this case, by the same argument,  
\begin{align*}
    &|\kappa|\leq 2 , 
    \\
    &|\alpha-\beta|+1\leq 2, 
\end{align*} 
giving us 

\begin{align*}
    \lVert \partial^\kappa\sigma 
    \partial^{\beta-\kappa}\sigma 
    \nabla\partial^{\alpha-\beta}\theta\rVert_{L^2} 
    &\leq\lVert\partial^\kappa\sigma\rVert_{L^\infty} 
    \lVert\partial^{\beta-\kappa}\sigma\rVert_{L^2} 
    \lVert\nabla\partial^{\alpha-\beta}\theta
    \rVert_{L^\infty} 
    \nonumber\\
    &\lesssim \lVert\sigma\rVert_{H^N}^2 
    \lVert \theta \rVert_{H^N} . 
\end{align*} 

\paragraph{Case 4:} $|\alpha-\beta|+1\geq N-1$ 
Again,  

\begin{align*}
    &|\kappa|\leq 2 , 
    \\
    &|\beta-\kappa|\leq 2, 
\end{align*} 
giving us 

\begin{align*}
    \lVert \partial^\kappa\sigma 
    \partial^{\beta-\kappa}\sigma 
    \nabla\partial^{\alpha-\beta}\theta\rVert_{L^2} 
    &\leq\lVert\partial^\kappa\sigma\rVert_{L^\infty} 
    \lVert\partial^{\beta-\kappa}\sigma\rVert_{L^\infty} 
    \lVert\nabla\partial^{\alpha-\beta}\theta\rVert_{L^2} 
    \nonumber\\
    &\lesssim \lVert\sigma\rVert_{H^N}^2 
    \lVert\theta\rVert_{H^N} 
    \lesssim\mathcal{E}_N(t)^{3/2}.  
\end{align*} 
Putting all of this together gives 

\begin{align*}
    \left\lVert 
    \sum_{0<\beta<\alpha} \binom{\alpha}{\beta} 
    \partial^\beta(\sigma^2)\nabla 
    \partial^{\alpha-\beta}\theta 
    \right\rVert_{L^2}\lesssim 
    \mathcal{E}_N(t)^{3/2}. 
\end{align*}
We finally examine the third term, 

\begin{align*}
    &\left\lVert\sum_{0<\beta<\alpha}\binom{\alpha}{\beta} 
    \partial^\beta(\bar h+\sigma) 
    \partial^{\alpha-\beta}(\bar h+\sigma) 
    \nabla \theta \right\rVert_{L^2}
    \nonumber\\
    &=\left\lVert\sum_{0<\beta<\alpha}
    \binom{\alpha}{\beta} 
    \left(\partial^\beta\bar h
    \partial^{\alpha-\beta}\bar h  
    +\partial^\beta\bar h
    \partial^{\alpha-\beta}\sigma
    +\partial^\beta\sigma
    \partial^{\alpha-\beta}\bar h
    +\partial^\beta\sigma
    \partial^{\alpha-\beta}\sigma\right)\nabla\theta 
    \right\rVert_{L^2}
    \nonumber\\
    &\lesssim (1+t)^{2B_k^*}\mathcal{E}_N(t)^{1/2} 
    +(1+t)^{B_k^*}\mathcal{E}_N(t) 
    +\mathcal{E}_N(t)^{3/2}. 
\end{align*}
Therefore, we have 

\begin{align*}
    \left\lVert\mathcal{J}_5^\alpha\right\rVert_{L^2}
    \lesssim (1+t)^{2B_k^*}\mathcal{E}_N(t)^{1/2}
    +(1+t)^{2B_k^*}\mathcal{E}_N(t) 
    +\mathcal{E}_N(t)^{3/2}. 
\end{align*}

\subsubsection{Bounding $\mathcal{J}^\alpha_6$}
By the same argument used for $\mathcal{J}_1^\alpha$, 

\begin{align*}
    &\lVert\mathcal{J}_6^\alpha\rVert_{L^2}
    \lesssim (1+t)^{B_k^*}
    \mathcal{E}_N(t)^{1/2}+\mathcal{E}_N(t). 
\end{align*}

\subsubsection{Bounds for $\mathcal{J}^\alpha_7-\mathcal{J}^\alpha_{10}$} 

The remaining commutator bounds for the isentropic case are estimated directly, 
\begin{align*}
    &\lVert\mathcal{J}^\alpha_7\rVert_{L^2} 
    \leq \sum_{0<\beta\leq\alpha}\binom{\alpha}{\beta} 
    \lVert\partial^{\alpha-\beta}\bfv\rVert_{L^2} 
    \lVert\nabla\partial^\beta\bar h\rVert_{L^\infty} 
    \lesssim (1+t)^{B_{k+1}^*}
    \mathcal{E}_N(t)^\frac{1}{2}, 
    \\
    &\lVert\mathcal{J}^\alpha_8\rVert_{L^2} 
    \leq \sum_{0<\beta\leq\alpha}\binom{\alpha}{\beta} 
    \lVert\partial^{\alpha-\beta}\bfv\rVert_{L^2} 
    \lVert\nabla\partial^\beta\bar \bfu\rVert_{L^\infty} 
    \lesssim (1+t)^{B_{k+1}^*}
    \mathcal{E}_N(t)^\frac{1}{2}, 
    \\
    &\lVert\mathcal{J}^\alpha_9\rVert_{L^2} 
    \leq \sum_{0<\beta\leq\alpha}\binom{\alpha}{\beta} 
    \lVert\partial^{\alpha-\beta}\sigma\rVert_{L^2} 
    \lVert\nabla\partial^\beta\bar \bfu\rVert_{L^\infty} 
    \lesssim (1+t)^{B_{k+1}^*}
    \mathcal{E}_N(t)^\frac{1}{2}, 
    \\
    &\lVert\mathcal{J}^\alpha_{10}\rVert_{L^2} 
    \leq \sum_{0<\beta\leq\alpha}\binom{\alpha}{\beta} 
    \lVert\partial^{\alpha-\beta}\sigma\rVert_{L^2} 
    \lVert\nabla\partial^\beta\bar h\rVert_{L^\infty} 
    \lesssim (1+t)^{B_{k+1}^*}
    \mathcal{E}_N(t)^\frac{1}{2}. 
\end{align*} 

\subsubsection{Bounds for $\mathcal{J}^\alpha_{11}-\mathcal{J}^\alpha_{13}$} 

Finally, we expand the last entropy-dependent terms which vanish identically in the isentropic case. 

\begin{align}
    \left\lVert\mathcal{J}_{11}^\alpha
    \right\rVert_{L^2}
    &\lesssim
    \left\lVert\nabla\bar S\right\rVert_{W^{k,\infty}}
    \left\lVert\sigma\right\rVert_{H^N}^2
    \lesssim
    \left\lVert\nabla\bar S\right\rVert_{W^{k,\infty}}
    \mathcal{E}_N(t),
    \label{Eqn:J11GeneralBound}
    \\
    \left\lVert\mathcal{J}_{12}^\alpha
    \right\rVert_{L^2}
    &\lesssim
    \left\lVert\bar h\nabla\bar S
    \right\rVert_{W^{k,\infty}}
    \left\lVert\sigma\right\rVert_{H^N}
    \lesssim
    \left\lVert\bar h\right\rVert_{W^{k,\infty}}
    \left\lVert\nabla\bar S\right\rVert_{W^{k,\infty}}
    \mathcal{E}_N(t)^{1/2},
    \label{Eqn:J12GeneralBound}
    \\
    \left\lVert\mathcal{J}_{13}^\alpha
    \right\rVert_{L^2}
    &\lesssim
    \left\lVert\nabla\bar S\right\rVert_{W^{k,\infty}}
    \left\lVert\bfv\right\rVert_{H^N}
    \lesssim
    \left\lVert\nabla\bar S\right\rVert_{W^{k,\infty}}
    \mathcal{E}_N(t)^{1/2}.
    \label{Eqn:J13GeneralBound}
\end{align}

\subsection{Derivation of Higher Order Energy Terms} 
\label{App:HigherOrderEnergy}

In this section, we manipulate our differentiated system in order to obtain the desired time evolution of the $k$-th order energy. Starting from 

\begin{align*}
    &(\partial^\alpha\sigma)_t
    +\bfU\cdot\nabla\partial^\alpha\sigma
    +aH\nabla\cdot\partial^\alpha\bfv=
    R^\alpha_1, 
    \\ 
    &(\partial^\alpha\bfv)_t
    +(\bfU\cdot\nabla)\partial^\alpha\bfv
    +\frac{1}{a}H\nabla\partial^\alpha\sigma 
    -\frac{1}{2a\gamma}H^2\nabla\partial^\alpha 
    \theta=R^\alpha_2 , 
    \\
    &(\partial^\alpha \theta)_t+\bfU\cdot\nabla 
    \partial^\alpha \theta = R^\alpha_3, 
\end{align*} 
we multiply the first equation by 

\begin{align*}
    \partial^\alpha\sigma-\frac{H}{2\gamma} 
    \partial^\alpha \theta, 
\end{align*} 
giving us 

\begin{align}
    \frac{1}{2}&((\partial^\alpha\sigma)^2)_t 
    -\frac{H}{2\gamma}\partial^\alpha\theta 
    (\partial^\alpha\sigma)_t 
    +\frac{1}{2}\bfU\cdot\nabla
    (\partial^\alpha\sigma)^2 
    -\frac{H}{2\gamma}\partial^\alpha\theta 
    \bfU\cdot\nabla\partial^\alpha\sigma 
    \nonumber\\
    &+aH\partial^\alpha\sigma\nabla\cdot
    \partial^\alpha\bfv-\frac{aH^2}{2\gamma}
    \partial^\alpha \theta\nabla\cdot\partial^\alpha 
    \bfv=
    \left(\partial^\alpha\sigma
    -\frac{H}{2\gamma}\partial^\alpha
    \theta\right)R^\alpha_1. 
    \label{Eqn:Sigma_Alpha_Int_0}
\end{align} 
We then multiply the second equation by $a^2\partial^\alpha\bfv$ to obtain  

\begin{align}
    \frac{a^2}{2}(|\partial^\alpha\bfv|^2)_t
    &+\frac{a^2}{2}(\bfU\cdot\nabla)
    |\partial^\alpha\bfv|^2
    +aH\partial^\alpha\bfv\cdot\nabla
    \partial^\alpha\sigma 
    \nonumber\\
    &-\frac{a}{2\gamma}H^2\partial^\alpha\bfv\cdot
    \nabla\partial^\alpha \theta
    =a^2\partial^\alpha\bfv\cdot R^\alpha_2. 
    \label{Eqn:V_Alpha_Int_0}
\end{align} 
Multiplying the third equation by 

\begin{align*}
    &\lambda\partial^\alpha \theta
    -\frac{H}{2\gamma}\partial^\alpha\sigma 
\end{align*} 
yields 

\begin{align}
    \frac{\lambda}{2}((\partial^\alpha 
    \theta)^2)_t 
    &-\frac{H}{2\gamma}\partial^\alpha\sigma 
    (\partial^\alpha \theta)_t 
    +\frac{\lambda}{2}\bfU\cdot \nabla 
    (\partial^\alpha \theta)^2 
    \nonumber\\
    &-\frac{H}{2\gamma}\partial^\alpha\sigma 
    \bfU\cdot\nabla\partial^\alpha \theta 
    =\left(\lambda\partial^\alpha \theta 
    -\frac{H}{2\gamma}\partial^\alpha\sigma
    \right)R_3^\alpha. 
    \label{Eqn:S_Alpha_Int_0}
\end{align} 
Adding (\ref{Eqn:Sigma_Alpha_Int_0})--(\ref{Eqn:S_Alpha_Int_0}) gives  

\begin{align}
    &\frac{1}{2}\partial_t\left( 
    (\partial^\alpha\sigma)^2+a^2
    |\partial^\alpha\bfv|^2 
    +\lambda(\partial^\alpha\theta)^2 
    -\frac{1}{\gamma} H
    \partial^\alpha\theta
    \partial^\alpha\sigma\right) 
    +\frac{1}{2\gamma}\partial^\alpha\theta
    \partial^\alpha\sigma H_t= 
    \nonumber\\
    &-\frac{1}{2}\left( 
    \bfU\cdot\nabla\left((\partial^\alpha\sigma)^2 
    +a^2|\partial^\alpha\bfv|^2 
    +\lambda(\partial^\alpha\theta)^2 
    \right)
    -\frac{H}{\gamma}\partial^\alpha\sigma 
    \bfU\cdot\nabla\partial^\alpha\theta 
    -\frac{H}{\gamma}\partial^\alpha\theta 
    \bfU\cdot\nabla\partial^\alpha\sigma
    \right) 
    \nonumber\\
    &-aH\partial^\alpha\sigma\nabla\cdot
    \partial^\alpha\bfv+\frac{aH^2}{2\gamma} 
    \partial^\alpha\theta\nabla\cdot\partial^\alpha\bfv 
    +\left(\partial^\alpha\sigma-\frac{H}{2\gamma} 
    \partial^\alpha\theta\right)R_1^\alpha 
    -aH\partial^\alpha\bfv\cdot\nabla
    \partial^\alpha\sigma 
    \nonumber\\
    &+\frac{a}{2\gamma}H^2\partial^\alpha\bfv 
    \cdot\nabla\partial^\alpha\theta 
    +a^2\partial^\alpha\bfv\cdot R_2^\alpha 
    +\left(\lambda\partial^\alpha\theta 
    -\frac{H}{2\gamma}\partial^\alpha\sigma
    \right)R_3^\alpha, 
    \label{Eqn:ptwseDensityUnsimplified}
\end{align} 
which, while complicated, contains the desired evolution of the pointwise energy density. We then move to simplify (\ref{Eqn:ptwseDensityUnsimplified}). First, note that 

\begin{align*}
    -\frac{1}{2\gamma}\partial^\alpha\theta 
    \partial^\alpha\sigma H_t 
    =\frac{1}{2\gamma}\partial^\alpha\theta 
    \partial^\alpha\sigma\left( 
    \bfU\cdot\nabla H+aH\nabla\cdot\bfU\right). 
\end{align*} 
Therefore, 

\begin{align*}
    &\frac{H}{2\gamma}\partial^\alpha\sigma 
    \bfU\cdot\nabla\partial^\alpha\theta 
    +\frac{H}{2\gamma}\partial^\alpha\theta 
    \bfU\cdot\nabla\partial^\alpha\sigma 
    +\frac{1}{2\gamma}\partial^\alpha\theta 
    \partial^\alpha\sigma\bfU\cdot\nabla H 
    =\frac{1}{2\gamma}\bfU\cdot\nabla(
    H\partial^\alpha\sigma\partial^\alpha\theta) , 
\end{align*}
and, 

\begin{align*}
    &-aH\partial^\alpha\sigma\nabla\cdot 
    \partial^\alpha\bfv+\frac{aH^2}{2\gamma} 
    \partial^\alpha\theta\nabla\cdot\partial^\alpha\bfv 
    -aH\partial^\alpha\bfv\cdot\nabla
    \partial^\alpha\sigma 
    +\frac{aH^2}{2\gamma}\partial^\alpha\bfv 
    \cdot\nabla\partial^\alpha\theta 
    \nonumber\\
    &=-a\nabla\cdot\left(H\left( 
    \partial^\alpha\sigma-\frac{H}{2\gamma}
    \partial^\alpha\theta\right)\partial^\alpha\bfv 
    \right)+a\left(\partial^\alpha\sigma 
    -\frac{H}{\gamma}\partial^\alpha\theta\right)
    \partial^\alpha\bfv\cdot\nabla H, 
\end{align*} 
which gives us 

\begin{align*}
    &\frac{1}{2}\partial_t e_\alpha 
    =-\frac{1}{2}\bfU\cdot\nabla e_\alpha 
    +\frac{aH}{2\gamma}\partial^\alpha\theta 
    \partial^\alpha\sigma\nabla\cdot\bfU 
    \nonumber\\
    &-a\nabla\cdot\left(H\left( 
    \partial^\alpha\sigma-\frac{H}{2\gamma}
    \partial^\alpha\theta\right)\partial^\alpha\bfv 
    \right)+a\left(\partial^\alpha\sigma 
    -\frac{H}{\gamma}\partial^\alpha\theta\right)
    \partial^\alpha\bfv\cdot\nabla H 
    \nonumber\\
    &+\left(\partial^\alpha\sigma-\frac{H}{2\gamma} 
    \partial^\alpha\theta\right)R_1^\alpha 
    +a^2\partial^\alpha\bfv\cdot R_2^\alpha 
    +\left(\lambda\partial^\alpha\theta 
    -\frac{H}{2\gamma}\partial^\alpha\sigma
    \right)R_3^\alpha . 
\end{align*}
Furthermore, Sobolev regularity of $(\sigma,\bfv,\theta)$ implies 

\begin{align*}
    -\frac{1}{2}\int_{\mathbb{R}^3}\bfU 
    \cdot\nabla e_\alpha \, dx = 
    \frac{1}{2}\int_{\mathbb{R}^3}e_\alpha 
    \nabla\cdot\bfU \, dx. 
\end{align*}
Therefore, 

\begin{align}
    &\frac{d}{dt} E_k(t) 
    = \frac{1}{2}\sum_{|\alpha|=k}\int_{\mathbb{R}^3} 
    e_\alpha \nabla\cdot\bfU \, dx
    +\sum_{|\alpha|=k}\biggl[\int_{\mathbb{R}^3}\frac{aH}{2\gamma}
    \partial^\alpha \theta 
    \partial^\alpha\sigma\nabla\cdot\bfU dx
    \nonumber\\
    &-a\int_{\mathbb{R}^3}\nabla\cdot\left(H\left( 
    \partial^\alpha\sigma-\frac{H}{2\gamma}
    \partial^\alpha\theta\right)\partial^\alpha\bfv 
    \right)dx
    \nonumber\\
    &+a\int_{\mathbb{R}^3}
    \left(\partial^\alpha\sigma 
    -\frac{H}{\gamma}\partial^\alpha\theta\right)
    \partial^\alpha\bfv\cdot\nabla H dx
    \nonumber\\
    &+\int_{\mathbb{R}^3}
    \left(\partial^\alpha\sigma-\frac{H}{2\gamma} 
    \partial^\alpha\theta\right)R_1^\alpha dx
    +a^2\int_{\mathbb{R}^3}
    \partial^\alpha\bfv\cdot R_2^\alpha dx
    \nonumber\\
    &+\int_{\mathbb{R}^3}\left(\lambda
    \partial^\alpha \theta 
    -\frac{H}{2\gamma}\partial^\alpha\sigma
    \right)R_3^\alpha dx\biggr]. 
    \label{Eqn:ddtEkUnsimplified}
\end{align} 
Using this, we bound each term in (\ref{Eqn:ddtEkUnsimplified}). 

\begin{align*}
    \left|\sum_{|\alpha|=k}\int_{\mathbb{R}^3} 
    e_\alpha\nabla\cdot\bfU dx \right|
    \lesssim \left(\lVert\nabla\bar\bfu 
    \rVert_{L^\infty}+\lVert\nabla\bfv 
    \rVert_{L^\infty}\right)E_k(t) . 
\end{align*} 
Then the next term, 

\begin{align*}
    \left|\int_{\mathbb{R}^3}H\partial^\alpha\theta
    \partial^\alpha\sigma\nabla\cdot\bfU dx\right| 
    \lesssim \left(\lVert\nabla\bar\bfu 
    \rVert_{L^\infty}+\lVert\nabla\bfv 
    \rVert_{L^\infty}\right)E_k(t). 
\end{align*} 
Based on the Sobolev regularity of $(\sigma,\bfv,\theta)$, the following term vanishes, 

\begin{align*}
    a\int_{\mathbb{R}^3}\nabla\cdot\left(H\left( 
    \partial^\alpha\sigma-\frac{H}{2\gamma}
    \partial^\alpha\theta\right)\partial^\alpha\bfv 
    \right)dx=0. 
\end{align*} 
The next term is then bounded as 

\begin{align*}
    \left|\int_{\mathbb{R}^3}
    \left(\partial^\alpha\sigma 
    -\frac{H}{\gamma}\partial^\alpha\theta\right)
    \partial^\alpha\bfv\cdot\nabla H dx \right| 
    \lesssim (\lVert \nabla \bar h \rVert_{L^\infty}
    +\lVert \nabla \sigma \rVert_{L^\infty}) 
    E_k(t). 
\end{align*} 
Then the last three terms:  

\begin{align*}
    \left|\int_{\mathbb{R}^3}
    \biggl(\partial^\alpha\sigma-\frac{H}{2\gamma} 
    \partial^\alpha\theta\biggr)R_1^\alpha dx \right| 
    \lesssim \biggl(&\lVert \nabla\sigma 
    \rVert_{L^\infty}+\lVert \nabla\bfv 
    \rVert_{L^\infty}+\lVert \nabla\bar\bfu 
    \rVert_{L^\infty}
    \nonumber\\
    +&\lVert \nabla\bar h 
    \rVert_{L^\infty}\biggr)E_k(t) 
    +\lVert\mathcal{F}^\alpha
    \rVert_{L^2}E_k(t)^{1/2}, 
\end{align*} 
and 

\begin{align*}
    \left|\int_{\mathbb{R}^3}
    \partial^\alpha\bfv\cdot R_2^\alpha dx\right| 
    \lesssim \biggl(&\lVert\nabla\bfv 
    \rVert_{L^\infty}+\lVert\nabla\sigma 
    \rVert_{L^\infty}+\lVert\nabla\theta 
    \rVert_{L^\infty}+\lVert\nabla\bar\bfu 
    \rVert_{L^\infty}\nonumber\\
    +&\lVert\nabla\bar h 
    \rVert_{L^\infty}
    +\lVert\nabla \bar S 
    \rVert_{L^\infty}\biggr)E_k(t)+
    \lVert\mathcal{G}^\alpha 
    \rVert_{L^2}E_k(t)^{1/2}, 
\end{align*} 
where our bootstrap assumption on $H$ is used. Then, 

\begin{align*}
    \left|\int_{\mathbb{R}^3}\left(\lambda
    \partial^\alpha \theta 
    -\frac{H}{2\gamma}\partial^\alpha\sigma
    \right)R_3^\alpha dx\right| 
    \lesssim\left(\lVert\nabla\theta\rVert_{L^\infty} 
    +\lVert\nabla\bar S\rVert_{L^\infty}\right)
    E_k(t)+\lVert\mathcal{H}^\alpha\rVert_{L^2} 
    E_k(t)^{1/2}. 
\end{align*} 
Adding these and summing over $|\alpha|=k$ gives us 

\begin{align}
    \frac{d}{dt}E_k(t)&\lesssim \biggl(\lVert\nabla\bfv 
    \rVert_{L^\infty}+\lVert\nabla\sigma 
    \rVert_{L^\infty}+\lVert\nabla\theta 
    \rVert_{L^\infty}+\lVert\nabla\bar\bfu 
    \rVert_{L^\infty}+\lVert\nabla\bar h 
    \rVert_{L^\infty}\nonumber\\
    +&\lVert\nabla \bar S 
    \rVert_{L^\infty}\biggr)E_k(t)
    +\sum_{|\alpha|=k}\left(
    \lVert\mathcal{F}^\alpha\rVert_{L^2} 
    +\lVert\mathcal{G}^\alpha\rVert_{L^2}
    +\lVert\mathcal{H}^\alpha\rVert_{L^2}\right) 
    E_k(t)^{1/2} . 
    \label{Eqn:ddtEkApp}
\end{align}

\end{appendices}

\section*{Acknowledgements} 

The author gratefully acknowledges support from the National Science Foundation under Award No. DMS-2306258 and from the University of Kansas General Research Fund. 

The author used OpenAI's ChatGPT for mathematical discussion, language editing, and assistance in checking calculations. All AI-assisted suggestions were independently reviewed and verified by the author, who takes full responsibility for the content of the paper. 

\bibliographystyle{abbrv}
\bibliography{references}

\end{document}